\pdfoutput=1
\PassOptionsToPackage{colorlinks,citecolor=blue}{hyperref}

\documentclass[
	english
]{scrartcl}
\usepackage[T1]{fontenc}
\usepackage[utf8]{inputenc}

\usepackage{amsmath,amsfonts,amsthm,amssymb}
\usepackage{xcolor}

\usepackage{preamble}
\usepackage{marginnote}

\usepackage{enumitem}
\setlist{itemsep=0em} % Condense lists
\setlist[enumerate]{label=(\roman*)}

\usepackage[capitalise,nameinlink]{cleveref}

\usepackage{dsfont}

\usepackage{soul,cancel}

\usepackage{graphicx}

\usepackage{microtype} % Takes care of some overful \hbox'es
\usepackage{lmodern}   % Otherwise, microtype does not work

\graphicspath{{pics/}}

\usepackage[ruled,vlined,linesnumbered]{algorithm2e}
\crefname{algocf}{Algorithm}{Algorithms}

\newif\ifbiber
\bibertrue
\ifbiber
\usepackage[
	style=authoryear-comp,
	sorting=nyt,
	url=false, 
	doi=true,                % Print the DOI.
	eprint=true,
	uniquename=allinit,        % Disambiguate between D and G Wachsmuth
	maxbibnames=99,          % Do not use ``et al'' in the bibliography -- now: only with 99+ authors ;)
	maxcitenames=2,
	uniquelist=minyear,
	dashed=false,            % Do not use ---- to replace the authors in the bibliography (if the previous entry has the same authors)
	sortcites=false,          % Do not sort if multiple keys in one \cite{}
	backend=biber,           % Use biber as backend
	safeinputenc,            % For some accents, see http://tex.stackexchange.com/questions/170562
]{biblatex}

\DeclareCiteCommand{\cite}{%
	\ifbibmacroundef{cite:init}{}{\usebibmacro{cite:init}}\usebibmacro{prenote}%
}{%
	\usebibmacro{citeindex}%
	\printtext[bibhyperref]{\usebibmacro{cite}}%
}{%
	\ifbibmacroundef{cite:init}{\multicitedelim}{}%
}{%
	\usebibmacro{postnote}%
}%
\DeclareCiteCommand{\parencite}[\mkbibbrackets]{%
	\ifbibmacroundef{cite:init}{}{\usebibmacro{cite:init}}\usebibmacro{prenote}%
}{%
	\usebibmacro{citeindex}%
	\printtext[bibhyperref]{\usebibmacro{cite}}%
}{%
	\ifbibmacroundef{cite:init}{\multicitedelim}{}%
}{%
	\usebibmacro{postnote}%
}%

\let\cite\parencite

\DeclareNameAlias{sortname}{last-first}
\DeclareNameAlias{default}{last-first}

\else

\usepackage{doi,natbib}

\fi

\newcommand\N{\mathbb{N}}
\newcommand\Z{\mathbb{Z}}
\newcommand\R{\mathbb{R}}

\newcommand\LL{\mathcal{L}}
\newcommand\VV{\mathcal{V}}

\renewcommand\d{\mathrm{d}}
\newcommand\dx{\d x}
\newcommand\ds{\d s}

\newcommand{\weakly}{\rightharpoonup}

\renewcommand\div{\operatorname{div}}

\DeclareMathAlphabet{\mathpzc}{OT1}{pzc}{m}{it}

\newcommand\bvarepsilon{{\boldsymbol{\varepsilon}}}
\DeclareMathSymbol{\dprod}{\mathbin}{operators}{"3A}
\newcommand\trace{\operatorname{trace}}

\newcommand\bi{\mathbf{i}}

\renewcommand\O{\Omega}
\newcommand\Oh{{\O_h}}
\newcommand\tu{\widetilde{u}}
\newcommand\hu{\widehat{u}}
\newcommand\veps{\varepsilon}
\newcommand\dist{\operatorname{dist}}
\newcommand\p{\partial}
\newcommand\cE{\mathcal{E}}
\newcommand\cH{\mathcal{H}}
\newcommand\cI{\mathcal{I}}
\newcommand\cS{\mathcal{S}}
\newcommand\cT{\mathcal{T}}

\newcommand\hO{\widehat{\O}}
\newcommand\hcT{\widehat{\cT}}
\newcommand\hcI{\widehat{\cI}}

\newcommand\dott[1]{{\buildrel{\hspace*{.03em}\text{\LARGE\textrm{.}}}\over{#1}}}

\newtheorem{theorem}{Theorem}[section]
\newtheorem{lemma}[theorem]{Lemma}
\newtheorem{proposition}[theorem]{Proposition}

\newtheorem{remark}[theorem]{Remark}

\newtheorem{example}[theorem]{Example}

\crefname{assumption}{Assumption}{Assumptions}

\newcommand{\sbb}[1]{#1} % {{\color{blue}#1}}

\newcommand\alert[1]{\textcolor{red}{\ifmmode #1 \else \textbf{#1}\fi}}

\definecolor{darkgreen}{rgb}{0,0.5,0}
\definecolor{darkred}{rgb}{0.8,0,0}

\begin{document}
%%fakesection: Title and co
\title{Numerical approximation of optimal convex shapes}

\author{%
	Sören Bartels%
	\footnote{%
                Albert-Ludwigs-Universit\"at Freiburg, 
                Department of Applied Mathematics, 
                79104 Freiburg i.Br., 
                Germany
		\url{https://aam.uni-freiburg.de/agba},
		\email{bartels@mathematik.uni-freiburg.de}%
	}
	\and
	Gerd Wachsmuth%
	\footnote{%
		Brandenburgische Technische Universität Cottbus-Senftenberg,
		Institute of Mathematics,
		% Chair of Optimal Control,
		03046 Cottbus,
		Germany,
		\url{https://www.b-tu.de/fg-optimale-steuerung},
		\email{gerd.wachsmuth@b-tu.de}%
	}%
	\orcid{0000-0002-3098-1503}%
}
 % \date{\today}
\publishers{}
 %\dedication{}
\maketitle

\begin{abstract}
This article investigates the numerical approximation of shape optimization problems 
with PDE constraint on classes of convex domains. The convexity 
constraint provides a compactness property which implies well posedness 
of the problem.
Moreover, we prove the convergence of discretizations in two-dimensional situations.
A numerical algorithm is devised that iteratively solves the discrete 
formulation. Numerical experiments show that optimal convex shapes are generally 
non-smooth and that three-dimensional problems require an appropriate relaxation of the
convexity condition. 
\end{abstract}
 
 
\begin{msc}
	\mscLink{49Q10}, \mscLink{65N12}
\end{msc}

\section{Introduction}
\label{sec:introduction} % \marginpar{possible journals: CoCV, M2AN, CMAM, BIT}
Shape optimization has become a popular area of research in applied mathematics
and is relevant in various technological applications. Existence theories and 
convergence results for numerical approximation methods often depend on
appropriate regularizations, e.g., via a perimeter functional. In this paper 
we consider PDE (partial differential equation) constrained shape optimization problems that are restricted
to classes of convex shapes. A convexity condition appears reasonable in many 
applications and may in some situations replace a more natural but mathematically 
more involved connectedness restriction. Imposing constraints on sets of admissible
shapes is often necessary to guarantee the existence of a solution, cf., e.g.,
\cite{BucurButtazzoNitsch2017,BartelsButtazzo2018} for a problem occurring in 
optimal insulation. 

We consider a model PDE constrained shape optimization problem with convexity 
constraint which reads as follows:
\begin{equation}
	\label{eq:prob}
	\tag{$\textbf{P}$}
	\begin{aligned}
		\text{Minimize}\quad & \int_\Omega j(x, u(x), \nabla u(x)) \, \dx \\
		\text{w.r.t.}\quad & \Omega \subset \R^d, \, u \in H_0^1(\Omega) \\
		\text{s.t.}\quad& -\Delta u = f \text{ in }\O \text{ and } \O \subset Q	\text{ convex and open}
	\end{aligned}
\end{equation}
Here, $Q \subset \R^d$ is a bounded, convex and open hold-all domain,
$j : Q \times \R \times \R^d \to \R$ is a suitable Carathéodory
function and the right-hand side in the state equation which is assumed to satisfy
$f \in L^2(Q)$. We stress that the quantity of interest is the a~priori unknown
shape $\Omega$. Particular attention is in this article paid to the convexity
constraint which enables us to prove existence of solutions and convergence
of approximations but which also leads to difficulties in its appropriate numerical
treatment. After establishing existence of solutions under suitable assumptions on
the function $j$, we prove convergence of numerical approximation schemes for
two-dimensional problems. We then devise an iterative scheme for computing 
optimal shapes and discuss numerical experiments which reveal that optimal convex
shapes are typically non-smooth and that the convexity constraint has to be relaxed
in three-dimensional situations. 

There are very few references discussing existence results for shape optimization 
problems with convexity constraints. To our knowledge, the first one is 
\cite{ButtazzoGuasoni1997} where 
the authors study the minimization of certain geometric functionals defined via
\begin{equation*}
	F(\Omega) = \int_{\partial\Omega} j(x, \nu(x))\,\d\cH^{d-1}.
\end{equation*}
Here, $\cH^{d-1}$ is the $(d-1)$-dimensional Hausdorff measure
and $\nu(x)$ is the (outer) normal vector of $\Omega$ at $x$
(which exists $\cH^{d-1}$-a.e.\ on $\partial\Omega$).
Existence results for shape optimization problems
involving convexity and PDE constraints
are given in \cite{VanGoethem2004,Yang2009}.
In the former contribution,
linear elliptic PDEs of even order are considered
and the latter reference addresses the situation of the stationary
Navier-Stokes equation under a smallness assumption on the data. None of these
articles addresses the practical solution of the problems. 

Another interesting result concerning convex shape optimization 
is provided in \cite{Bucur2003} in which regularity of optimal convex shapes 
for a class of objectives is proved.
In particular, the author considers problem \eqref{eq:prob} with the objective
\begin{equation*}
	\int_Q j(x, u(x)) \, \dx + \alpha \, \int_\Omega 1 \, \dx
\end{equation*}
for some positive regularization parameter $\alpha > 0$.
Note that the first integral ranges over the hold-all domain $Q$.
It is shown that the optimal shape has a $C^1$ boundary
under some weak assumptions,
namely it is required that $j$ is Lipschitz continuous in its second argument
and that $f \in L^\infty(Q)$. Under these conditions the objective value can
be decreased by rounding corners. However, the arguments do not apply to the case
of a vanishing regularization parameter $\alpha = 0$
or if $j$ is integrated only over $\Omega$ which is the situation considered in
this article. 

To our knowledge, problem \eqref{eq:prob} has not been the subject of a 
rigorous numerical analysis yet although some computational schemes have been
devised in the literature. In \cite{LachandRobertOudet2005} the authors
propose an algorithm for minimizing functionals
over sets of convex bodies. Similar to the setting in \cite{ButtazzoGuasoni1997},
they consider only geometric functionals of the form
\begin{equation*}
	F(\Omega) = \int_{\partial\Omega} j(x, \nu(x), \varphi(x))\,\d\cH^{d-1},
\end{equation*}
where $\varphi(x)$ is the signed distance of the supporting hyperplane of $\Omega$ at $x$
to the origin $0$.
The proposed algorithm is based on approximating convex bodies
by the intersection of finitely many half spaces.
Since the topology of this intersection changes throughout the algorithm,
it is not clear how this algorithm can be coupled with a (finite element) discretization 
of a PDE. In the recent preprint \cite{AntunesBogosel2018},
the authors approximate the support function of a convex body by
its Fourier series decomposition ($d = 2$) and by its
spherical harmonic decomposition ($d = 3$). The convexity amounts to an inequality constraint
on the unit sphere $S^{d-1} \subset \R^d$. This convexity constraint is discretized
by its sampling in a finite number of points. 

Other related contributions address optimization problems over classes of convex functions.
The developed methods are of interest for the treatment of problem \eqref{eq:prob} if 
admissible convex shapes can be represented via superlevel sets of convex functions over a
fixed domain, i.e., if they are of the form $\O = \{(x',x_n) \in D \times [-M,0] : \phi(x') \le x_n \le 0\}$
with a convex set $D \subset \R^{d-1}$ and a convex function 
$\phi : D \to [-M, 0]$ for some  $M > 0$. For various approaches to solve optimization 
problems over sets of convex functions, we refer to \cite{Mirebeau2016}
and the references therein.

The outline of this article is as follows. In \cref{sec:prelim} we review
formulas for shape derivatives in the absence of a convexity constraint 
which are important for the formulation of our iterative algorithms. 
\cref{sec:existence}
is devoted to a general existence result which serves as a template for the 
convergence of discretizations which are discussed in \cref{sec:convergence}.
Our iterative algorithm is formulated in \cref{sec:numerics}. Its 
performance and qualitative features of optimal convex shapes are illustrated
in \cref{sec:num_ex}. 

\section{Shape derivatives}
\label{sec:prelim}

In this section we discuss different representations of shape derivatives
for PDE constrained shape optimization problems in the absence of a convexity
constraint. For this, we define the shape functional 
\begin{equation*}
	J(\Omega) :=
	\int_\Omega j(x, u(x), \nabla u(x)) \, \dx,
\end{equation*}
where $u \in H_0^1(\Omega)$ is the solution of the state equation in \eqref{eq:prob}
on $\Omega$, i.e., $u$ is the weak solution of $-\Delta u =f$ in $\O$ subject
to homogeneous Dirichlet boundary conditions. 

\subsection{Weak shape derivative}
\label{subsec:weak_derivative}
Let $\Omega \subset Q$ be a fixed convex and open domain.
We derive expressions for the shape derivative of $J$ at $\Omega$.
We first use the method of \emph{perturbation of identity}.
Since we are only interested in first-order derivatives,
the same expressions are obtained by the \emph{speed method}.
Let a vector field $V \in C^{0,1}_c(Q)$ be given.
For small $t \ge 0$, we introduce the perturbations of the identity
\begin{equation*}
	T_t(x)
	:=
	x + t \, V(x)
	.
\end{equation*}
If $t$ is small enough, $T_t : \bar Q \to \bar Q$
is a bijective Lipschitz mapping with a Lipschitz continuous inverse.
This enables us to define the family of perturbed domains
\begin{equation*}
	\Omega_t
	=
	T_t(\Omega).
\end{equation*}
The state associated with $\Omega_t$ is denoted by $u_t \in H_0^1(\Omega_t)$.
Thus,
\begin{equation*}
	J(\Omega_t)
	=
	\int_{\Omega_t} j(x, u_t(x), \nabla u_t(x)) \, \dx.
\end{equation*}
We are interested in giving an expression for the
\emph{Eulerian derivative}
\begin{equation*}
	J'(\Omega; V)
	:=
	\lim_{t \searrow 0}
	\frac{J(\Omega_t) - J(\Omega)}{t}.
\end{equation*}
We argue formally and note that all the computations can be made rigorous
by imposing appropriate smoothness assumptions on $j$.
To calculate the shape derivative $J(\Omega; V)$,
we use a chain-rule approach utilizing a formula for the
material derivative of the states.
The material derivative $\dott u \in H_0^1(\Omega)$
is defined as
\begin{equation*}
	\dott u
	=
	\lim_{t \searrow 0} \frac{u^t - u}{t}
	\qquad\text{in } H_0^1(\Omega),
\end{equation*}
where $u \in H_0^1(\Omega)$ is the state on $\Omega$
and $u^t = u_t \circ T_t$ is the pull-back of $u_t$.
Note that $u^t \in H_0^1(\Omega)$, see, e.g.,
\cite[Thm.~2.2.2]{Ziemer1989}.
It is well known that $\dott u$ satisfies the variational equation
\begin{equation}
	\label{eq:material_derivative}
	\int_\Omega \nabla \dott u \cdot \nabla w \,\dx
	=
	\int_\Omega \nabla u^\top [ DV + DV^\top - \div(V) \, I]\nabla w + \div(f \, V) \, w \,\dx
\end{equation}
for all $w\in H^1_0(\O)$. Here, $DV$ is the Jacobian of $V$.
Next, we use a substitution rule to obtain
\begin{align*}
	J(\Omega_t)
	&=
	\int_{\Omega_t} j(x, u_t(x), \nabla u_t(x)) \, \dx.
	\\
	&=
	\int_\Omega j(T_t(x), u^t(x), DT_t^{-\top} \nabla u^t(x)) \, \det(DT_t) \, \dx.
\end{align*}
Thus, a (formal) application of the chain rule leads to
\begin{equation*}
	J'(\Omega; V)
	=
	\int_\Omega
	j_x(\cdot) \cdot V
	+
	j_u(\cdot) \, \dott u
	+
	j_v(\cdot) \cdot [\nabla \dott u - DV^\top \nabla u]
	+
	j(\cdot) \, \div(V)
	\,\dx.
\end{equation*}
Here,
$j_x(\cdot)$, $j_u(\cdot)$ and $j_v(\cdot)$ are the partial derivatives of $j$ w.r.t.\ its three arguments.
Moreover, we have abbreviated the point
$(x, u(x), \nabla u(x))$
by the placeholder $(\cdot)$.
In order to get rid of the material derivative $\dott u$,
we introduce the adjoint state $p \in H_0^1(\Omega)$
as the solution of
\begin{equation}
	\label{eq:adjoint_equation}
	\int_\Omega \nabla p\cdot \nabla w \, \dx
	=
	-\int_\Omega j_u(\cdot) \, w + j_v(\cdot) \cdot \nabla w \, \dx
\end{equation}
for all $w \in H_0^1(\Omega)$.
Using the test function $w = \dott u$ in \eqref{eq:adjoint_equation}
and $w = p$ in \eqref{eq:material_derivative}, we obtain our final expression
for the shape derivative
\begin{equation}
	\label{eq:shape_derivative}
	\begin{aligned}
	J'(\Omega; V)
	=
	\int_\Omega
	&j_x(\cdot) \cdot V
	-
	j_v(\cdot) \cdot DV^\top \nabla u
	+
	\nabla p^\top \bigl[-DV - DV^\top + \div(V) \,I\bigr]\nabla u
	\\
	&\qquad
	-
	\div(f \, V) \, p
	+
	j(\cdot) \, \div(V)
	\,\dx
	.
	\end{aligned}
\end{equation}
It is also known that a very similar expression
can be obtained for finite-element discretizations
of the shape functional,
see, e.g.,
\cite{DelfourPayreZolesio1985}
or \cite[Remark~2.3 in Chapter~10]{DelfourZolesio2011}.

\subsection{Hadamard derivative}
\label{subsec:hadamard_derivative}
Next, we are going to derive the shape derivative
in the Hadamard form, i.e., in the form
\begin{equation*}
	J'(\O; V)
	=
	\int_{\partial\O} g \, (V \cdot n) \, \ds
\end{equation*}
with an appropriate function $g$ on the boundary.
To this end, we fix a domain $\Omega$ with a $C^2$ boundary.
We follow~\cite[Section~6.2 in Chapter~10]{DelfourZolesio2011} and rewrite the 
minimization of 
\[
J(\O) = \int_\O j(x,u(x),\nabla u(x)) \,\dx 
\]
subject to the constraint 
\[
-\Delta u = f \text{ in }\O, \quad u|_{\p\O} = 0,
\]
as a saddle-point problem.
Therefore, we introduce the functional
\[
	L(\O,u,p) = \int_\O (-\Delta u - f) \, p \,\dx  - \int_\O \div (u \nabla p) \,\dx
\]
to incorporate the PDE constraint. Here, $u, p \in H^2(\Omega)$. The Lagrangian is given by
\begin{equation*}
	G(\O, u, p)
	=
	\int_\O j(x,u(x),\nabla u(x))\,\dx + L(\O,u,p).
\end{equation*}
Now, we check that
\begin{equation}
	\label{eq:saddle-point}
	J(\O) = \min_{u \in H^2(\O)} \max_{p \in H^2(\O)} G(\O,u,p)
	.
\end{equation}
An integration by parts yields that
\begin{equation*}
	L(\O,u,p)
	=
	\int_\O (-\Delta u - f) \, p \,\dx
	-
	\int_{\partial\O} u \, \frac{\partial p}{\partial n} \,\dx
	.
\end{equation*}
Using localized test functions $p \in H^2(\Omega)$,
it follows that
\begin{equation*}
	\max_{p \in H^2(\O)} G(\O,u,p)
	=
	\begin{cases}
		\int_\O j(x,u(x),\nabla u(x))\,\dx
		&\text{if } {-\Delta u} = f \text{ in } \Omega
		\text{ and } u|_{\p\O} = 0,
		\\
		+\infty
		&\text{else}.
	\end{cases}
\end{equation*}
This justifies \eqref{eq:saddle-point}.
Hence, the first component~$u$ of a saddle point $(u,p) \in H^2(\Omega)^2$ satisfies
\begin{equation}\label{eq:hadamard_primal}
	{-\Delta u} = f \text{ in } \O, \quad u|_{\p\O} = 0.
\end{equation}
The second component $p$
satisfies
\[
\int_\O j_u(x,u,\nabla u) v + j_v (x,u,\nabla u)\cdot\nabla v 
+ (-\Delta v) \, p - \div (v \nabla p) \,\dx = 0
\]
for all test functions $v \in H^2(\Omega)$. After applying integration by parts
to the term $(\Delta v) p$
and by using the product rule for the last term,
we find that in strong form we have
\begin{equation}
	\label{eq:hadamard_adjoint}
	{-\Delta p} = -j_u(x,u,\nabla u) + \div  j_v(x,u,\nabla u) \text{ in } \O, \quad
	p|_{\p\O} = 0.
\end{equation}
We extend the saddle point problem to the whole space $\R^d$ using the sets
\[
X(t) = \big\{U \in H^2(\R^d): U|_{\O_t} = u_t\big\},  \quad 
Y(t) = \big\{P \in H^2(\R^d): P|_{\O_t} = p_t\big\},
\]
where for a given set $\O_t$ the functions $u_t$ and $p_t$ 
solve~\eqref{eq:hadamard_primal} and~\eqref{eq:hadamard_adjoint} in~$\O_t$.
Formally, assuming the existence of a sufficiently regular saddle point,
it follows from a theorem by Correa and Seeger, see
\cite[Thm.~5.1, p.~556]{DelfourZolesio2011}, for the shape derivative of $J$
given a velocity field $V$ that 
\[
dJ(\O;V) = \min_{U\in X(0)} \max_{P \in Y(0)} \p_t G(\O_t,U,P)|_{t=0}.
\]
In particular, it follows with $u = u_0$ and $p=p_0$ and $\O = \O_0$ that 
\[
dJ(\O;V) = \int_{\p\O} \Big\{j(x,u(x),\nabla u(x)) 
- (\Delta u + f) p - \div(u\nabla p) \Big\} \, (V\cdot n) \,\ds.
\]
Incorporating the characterizations of $u$ and $p$, i.e., $-\Delta u = f$, $u=0$,
and $p=0$ on $\p\O$ this reduces to 
\[
	dJ(\O;V) = \int_{\p\O} \Big\{j(x,u(x),\nabla u(x)) -  \frac{\p u}{\p n} \, \frac{\p p}{\p n} \Big\} V\cdot n \,\ds,
\]
since tangential derivatives of $p$ on $\p\O$ vanish and hence $\nabla p = (\p p/\p n) \, n$.

\section{Existence of optimal convex shapes}
\label{sec:existence}
For convenience, we repeat the PDE constrained shape optimization problem
formulated in the introduction:
\begin{equation}
	\label{eq:prob_2}
	\tag{$\textbf{P}$}
	\begin{aligned}
		\text{Minimize}\quad & \int_\O j(x, u(x), \nabla u(x)) \, \dx \\
		\text{w.r.t.}\quad & \O \subset \R^d, u \in H_0^1(\O) \\
		\text{s.t.}\quad
		& -\Delta u = f\text{ in }\O \text{ and }\O \subset Q \text{ convex and open }
	\end{aligned}
\end{equation}
Here, $Q\subset \R^d$ is bounded and convex and $f\in L^2(Q)$. Conditions on 
$j:Q \times \R\times \R^d \to \R$ are formulated below. We remark that the empty 
set $\Omega = \emptyset$ is an admissible point of \eqref{eq:prob} with objective value $0$.

\begin{proposition}[Existence]\label{prop:existence}
Assume that $j$ is a Carathéodory function
(i.e., measurable in the first and continuous in its remaining
arguments) and that there exist $a\in L^1(Q)$ and $c\ge 0$ such that 
\[
|j(x,s,z)| \le a(x) + c (|s|^p + |z|^2)
\]
with $p=2d/(d-2)$ if $d >2$ and $p<\infty$ if $d=2$. 
% constants $c_1,c_2\ge 0$ and $r\ge 0$ such that
% \[
% \Big|\int_Q j(x,v(x),\nabla v(x)) \,\dx \Big| \le c_1 + c_2 \|\nabla v\|_{L^2(Q)}^r
% \]
% holds for all $v\in H^1_0(Q)$. 
Then there exists an optimal pair $(\O,u)$ for \eqref{eq:prob_2}.
\end{proposition}

\begin{proof}
\emph{(i) Choice of an infimizing sequence.} 
Noting that $\|\nabla \tu \|_{L^2(Q)}\le \|\nabla u\|_{L^2(\O)} \le c_P \|f\|_{L^2(Q)}$
holds for every convex domain $\O$ and corresponding solution $u\in H^1_0(\O)$ with trivial
extension $\tu$ to $Q$ of the
Poisson problem, we have that the objective function is 
bounded from below on the set of admissible pairs $(\O,u)$. 
We may thus select an infimizing sequence $(\O_n,u_n)$.

\emph{(ii) Existence of accumulation points.}
Using Lemma~3.1 from~\cite{ButtazzoGuasoni1997} we find that there exists
an open convex set $\O\subset Q$ and a subsequence of $(\O_n)_n$ 
(which is not relabeled in what follows) such that the sequence
of distributional derivatives of the 
characteristic functions $\chi_{\O_n}$ converges in variation to $\chi_\O$.  
In particular, we have by the compact embedding of $BV(\O)$ into $L^1(\O)$
(see, e.g., Theorem~10.1.4 in~\cite{AttouchButtazzoMichaille2014})
that the characteristic functions $\chi_{\O_n}$ converge
strongly in $L^1(Q)$ to $\chi_\O$.
Moreover, it follows from Lemma~4.2 in~\cite{ButtazzoGuasoni1997} that for 
every $\veps>0$ there exists $N>0$ such that for the symmetric set difference
\[
\O_n \bigtriangleup \O = (\O\setminus \O_n) \cup (\O_n\setminus \O)
\]
we have for all $n\ge N$ that 
\begin{equation}\label{eq:uniform_conv_sets}
\O_n \bigtriangleup \O \subset \{x\in Q : \dist(x,\p \O) \le \veps \}.
\end{equation}
We trivially extend the solutions $u_n \in H^1_0(\O_n)$ to functions 
$\tu_n \in H^1_0(Q)$. Since this defines a bounded sequence in $H^1_0(Q)$ 
we may extract a weakly convergent subsequence with limit $\tu \in H^1_0(Q)$. 
We have to show that $\tu|_\O$ solves the Poisson problem on $\O$. For this
let $\phi\in C^1(Q)$ be compactly supported in $\O$.
For $n\ge N$ with $N$ sufficiently large we deduce from~\eqref{eq:uniform_conv_sets}
that $\phi\in C^1_c(\O_n)$. Hence, it follows that
\[
\int_\O \nabla u \cdot \nabla \phi \,\dx 
= \lim_{n\to \infty} \int_Q \nabla \tu_n \cdot \nabla \phi \,\dx 
= \lim_{n\to \infty} \int_Q f \phi \,\dx 
= \int_\O f \phi \,\dx.
\]
It remains to show that
$u := \tu|_\O$
belongs to $H_0^1(\Omega)$.
To this end,
let $A$ be a compact subset in 
$Q\setminus \overline{\O}$. For $n\ge N$ with $N$ sufficiently large we 
have $\tu_n|_A = 0$ and hence
\[
\|\tu \|_{L^2(A)} = \|\tu-\tu_n\|_{L^2(A)} \le \|\tu - \tu_n\|_{L^2(Q)} \to 0
\]
as $n\to \infty$. Hence $\tu|_A =0$ for every compact subset 
$A\subset Q\setminus \overline{\O}$. We may therefore approximate $u$ by
a sequence of smooth functions that are compactly supported in $\O$ 
and this implies $u|_{\p\O}=0$, i.e., $u\in H^1_0(\O)$.

\emph{(iii) Optimality of the limit.} To show that the constructed pair
$(\O,u)$ solves the optimization problem we first note that the (sub-) sequence
$(\tu_n)$ converges strongly to $\tu$ in $H^1_0(Q)$. This is an immediate consequence
of the identities
\[
\int_Q |\nabla \tu|^2 \,\dx = \int_Q f u \,\dx 
= \lim_{n \to \infty}\int_Q f \tu_n \,\dx = \lim_{n \to \infty} \int_Q |\nabla \tu_n|^2 \,\dx.
\]
With the assumptions on $j$ it follows from an application of Fatou's lemma
that the objective functional is $\gamma$-continuous, see
\cite[p.~213]{AttouchButtazzoMichaille2014} for details, i.e., that
\[
\int_{\O_n} j(x,u_n(x),\nabla u_n(x))\,\dx \to \int_\O j(x,u(x),\nabla u(x))\,\dx,
\]
which implies that the pair $(\O,u)$ solves the shape optimization problem. 
\end{proof}

We remark that for a simpler class of objective functionals, a related existence result 
can be found in \cite[Thm.~6.2 in Chapter 6]{DelfourZolesio2001}. Generalizations
of \cref{prop:existence} can be made concerning boundary terms. 

\begin{remark}[Boundary terms]
Noting that the sequence of convex sets $(\O_n)$ converges in variation to $\O$
we may incorporate a boundary term
\[
\int_{\p\O} g(x,\nu(x)) d \mathcal{H}^{n-1} = \int_Q f(x,\nu) \d |D\chi_\O|
\]
in the objective functional, cf.~\cite{ButtazzoGuasoni1997}.
\end{remark}

The following example shows that shape optimization problems are often ill posed
if the class of admissible domains is too large.

\begin{example}[Non-existence without convexity assumption]	
	To construct our counterexample,
	we use the classical construction from \cite{CioranescuMurat1997}.
	Let an arbitrary bounded and open set $Q \subset \R^d$ be given.
	We construct a perforated domain, as described in \cite[Ex.~2.1]{CioranescuMurat1997}.
	That is, we choose a sequence $\{\varepsilon_n\}_{n \in \N} \subset (0,\infty)$ 
        with $\varepsilon_n \to 0$ and set
	\begin{equation*}
		r_n =
		\begin{cases}
			\exp^{-\varepsilon_n^{-2}} & \text{if } d = 2, \\
			\varepsilon_n^{d / (d-2)} & \text{if } d > 2.
		\end{cases}
	\end{equation*}
	For each $\bi \in \Z^d$, let $T_\bi^n = B_{r_n}(\varepsilon_n \, \bi)$ be the closed ball with radius $r_n$ centered at $\varepsilon_n \, \bi$.
	Now, the perforated domain is given by
	\begin{equation*}
		\Omega_n = Q \setminus \bigcup_{\bi \in \Z^d} T_\bi^n.
	\end{equation*}
	As $n \to \infty$, both the distance $\varepsilon_n$ and the radius $r_n$ of the holes
	go to $0$.
	Now we define $u_n \in H_0^1(\Omega_n)$ as the weak solution of
	\begin{equation*}
		-\Delta u_n = 1 \quad \text{in } \Omega_n,
		\qquad
		u_n = 0 \quad \text{on } \partial\Omega_n,
	\end{equation*}
	and extend $u_n$ by $0$ to a function in $H_0^1(Q)$.
	By \cite[Thms.~1.2, 2.2]{CioranescuMurat1997}, $u_n$ converges weakly in $H_0^1(Q)$ to
	the weak solution $\hat u \in H_0^1(Q)$ of
	\begin{equation*}
		-\Delta \hat u + \mu \, \hat u = 1 \quad \text{in } Q,
		\qquad
		\hat u = 0 \quad \text{on } \partial Q,
	\end{equation*}
	for some $\mu > 0$.
	For the precise value of $\mu$, we refer to \cite[eq.~(2.3)]{CioranescuMurat1997}.
	After this preparation, we choose the objective
	\begin{equation*}
		j(x, u, g)
		:=
		-1
		+
		(u - \hat u(x))^2.
	\end{equation*}
	Let us check that
	for an arbitrary open set $\Omega \subset Q$ with associated state $u$,
	we have
	\begin{equation*}
		\int_\Omega
		j(x, u(x), \nabla u(x))
		\,\dx
		> - \LL^d(Q).
	\end{equation*}
	Here, $\LL^d$ is the $d$-dimensional Lebesgue measure.
	Indeed,
	it is clear that
	$\LL^d(\Omega) \le \LL^d(Q)$.
	Moreover, one can check that $u \ne \hat u$ in $L^2(\Omega)$.
	This follows, e.g.,
	from inner regularity and
	\begin{equation*}
		-\Delta u(x) = 1
		\ne
		1 - \mu \, \hat u(x) = -\Delta \hat u(x)
	\end{equation*}
	for all inner points $x \in \Omega$.
	On the other hand,
	the sequence $\Omega_n$ together with the associated states $u_n$
	satisfies
	\begin{equation*}
		\int_{\Omega_n}
		j(x, u_n(x), \nabla u_n(x))
		\,\dx
		=
		-\LL^d(\Omega_n)
		+
		\int_{\Omega_n} (u_n - \hat u)^2 \, \dx
		\to
		-\LL^d(Q),
	\end{equation*}
	since the measure of the holes tend to zero
	and
	$u_n \weakly \hat u$ in $H_0^1(Q)$
	implies
	$u \to \hat u$ in $L^2(Q)$.
	Thus, the infimal value $-\LL^d(Q)$
	is not attained.
\end{example}

\section{Convergence of discretizations}
\label{sec:convergence}

In this section we discuss convergence results for suitable discretizations
of problem~\eqref{eq:prob}. The first one is a general statement under moderate
assumptions while the second one requires further conditions but serves as the
basis for the iterative scheme devised in \cref{sec:numerics} below. 

\subsection{Abstract convergence analysis}
For a universal constant $c_{\mathrm{usr}}>0$ we consider the class $\mathbb{T}_{c_{\mathrm{usr}}}$
of conforming, uniformly shape regular triangulations $\cT_h$ of polyhedral
subsets of $\R^d$ such that 
$h_T/\varrho_T \le c_{\mathrm{usr}}$ for all elements $T\in \cT_h$ with diameter $h_T\le h$ and 
inner radius $\varrho_T$. For a discretization fineness $h>0$ we consider the 
following discrete version of~\eqref{eq:prob}:
\begin{equation}
	\label{eq:prob_2_discr}
	\tag{$\textbf{P}_h$}
	\begin{aligned}
		\text{Minimize}\quad & \int_\Oh j(x, u_h(x), \nabla u_h(x)) \, \dx \\
		\text{w.r.t.}\quad & \Oh \subset \R^d, \, 
                      \cT_h \in \mathbb{T}_{c_{\mathrm{usr}}} \text{ triangulation of } \Oh, \, u_h \in \cS^1_0(\cT_h) \\
		\text{s.t.}\quad
		& -\Delta_h u_h = f_h \text{ in }\Oh \text{ and }\Oh \subset Q \text{ convex and open }
	\end{aligned}
\end{equation}
Here, $\cS^1_0(\cT_h)\subset H^1_0(\Oh)$ consists of all piecewise affine, globally
continuous functions in $H^1_0(\Oh)$. The identity $-\Delta_h u_h = f_h$ represents the 
discrete formulation of the Poisson problem, i.e., $u_h\in \cS^1_0(\cT_h)$ satisfies 
\[
\int_\Oh \nabla u_h \cdot \nabla v_h \,\dx = \int_\Oh f v_h \, \dx
\]
for all $v_h \in \cS^1_0(\cT_h)$. Note that the formulation $\eqref{eq:prob_2_discr}$ may not admit
a minimizer since triangulations corresponding to infimizing sequences may not have 
admissible accumulation points
as the number of nodes in the corresponding triangulations might be unbounded.
However, for the statements of the following results only the existence of almost-infimizing
points is necessary. Further constraints such as a bound on the number of elements in $\cT_h$,
e.g., $\#\cT_h\le c h^{-d}$ can be included to obtain a more practical minimization problem
that admits a solution due to finite dimensionality. \sbb{The following estimate provides
an approximation result for certain regular solutions and will be used to justify the general
convergence of the discretization below.}

\begin{proposition}[Consistency]\label{prop:consistency}
Assume that $d=2$. \sbb{Let $(\O,u)$ be such that $\O$ is convex with piecewise 
$C^2$ regular boundary, that the solution $u\in H^1_0(\O)$ of the corresponding
Poisson problem in $\O$ satisfies 
$u\in H^2(\O)\cap W^{1,\infty}(\O)$}, and assume that $j$ is such that 
the functional $J$ defined by 
\[
J(A,v) = \int_A j(x,v(x),\nabla v(x))\, \dx
\]
satisfies for some constant $c_J \ge 0$ the estimate 
\[
\big|J(A,v) - J(A,w)\big| \le  
c_J (1+\|\nabla v\|_{L^2(A)} + \|\nabla w\|_{L^2(A)}) \|\nabla (v-w)\|_{L^2(A)}
\]
for every open set $A\subset Q$ and $v,w\in H^1_0(A)$.
Then there exist admissible tuples $(\Oh,\cT_h,u_h)$ for~\eqref{eq:prob_2_discr} and 
$(\Oh,\hu^{(h)})$ for~\eqref{eq:prob} such that $\Oh \subset \O$, 
\[
\Oh \bigtriangleup \O \subset \{x\in \O: \dist(x,\p\O) \le c h^2\},
\]
and
\[
\|\nabla (u-\tu_h)\|_{L^2(\O)} + \|\nabla (u_h-\hu^{(h)})\|_{L^2(\Oh)} \le c h.
\]
In particular, we have
\[
0 \le J(\Oh,\hu^{(h)}) - J(\O,u) \le c h.
\]
\end{proposition}

\begin{proof}
We first choose an interpolating triangulation $\cT_h$ of $\O$ with 
maximal mesh-size $h>0$. In the two-dimensional situation under consideration, 
the interior of the union of elements in $\cT_h$ 
defines a convex open set $\Oh \subset \O$.
Standard results on boundary approximation, cf., e.g.,~\cite{Dziuk2010} 
and~\cite[Prop. 3.7]{Bart16},
yield that the corresponding finite element approximation of $u$ satisfies the 
asserted estimate. The function
\[
\hu^{(h)} = u_h + r^{(h)}
\]
is defined with the correction $r^{(h)} \in H^1_0(\Oh)$ that satisfies
\[
\int_\Oh \nabla r^{(h)} \cdot \nabla v \,\dx
= \int_\Oh f v \,\dx - \int_\Oh \nabla u_h \cdot \nabla v \,\dx
\]
for all $v\in H^1_0(\Oh)$ so that we have $-\Delta \hu^{(h)} = f$ in 
$\Oh$. It follows that 
\[\begin{split}
\|\nabla r^{(h)} \|_{L^2(\Oh)}^2 
&= \int_\Oh f \, r^{(h)} \,\dx - \int_\Oh \nabla u_h \cdot \nabla r^{(h)} \,\dx \\
&= \int_\Oh \nabla u \cdot \nabla r^{(h)}\,\dx - \int_\Oh \nabla u_h \cdot \nabla r^{(h)} \,\dx \\
&\le \|\nabla (u-u_h)\|_{L^2(\Oh)} \|\nabla r^{(h)}\|_{L^2(\Oh)}.
\end{split}\]
This implies the result.
\end{proof}

\begin{remark}
Note that an interpolating polygonal domain of a convex domain is in general not
convex if $d=3$ so that the arguments of the proof cannot be directly
generalized to that setting.
\end{remark}

The second result concerns the stability of the method, i.e., that
solutions for~\eqref{eq:prob_2_discr} accumulate at solutions for~\eqref{eq:prob}
as the mesh size tends to zero. 

\begin{proposition}[Stability]\label{prop:stability}
Let $d=2$ and assume that $j$ satisfies the conditions of \sbb{\cref{prop:existence}
and~\cref{prop:consistency}}
and let $(\Oh,\cT_h,u_h)_{h>0}$ be a sequence of discrete solutions for~\eqref{eq:prob_2_discr}. 
Then, every accumulation pair $(\O,\cT)$ solves~\eqref{eq:prob}. 
\end{proposition}

\begin{proof}
We argue as in the proof of \cref{prop:existence}. Accumulation
points $\O$ of the sequence of convex sets $(\Oh)$ are convex and the finite
element solutions are bounded in $H^1_0(Q)$ with weak accumulation points $u\in H^1_0(\O)$.
To show that every such point solves the Poisson problem in $\O$ we choose a
smooth, compactly supported function $\phi\in C^\infty_c(\O)$ and note that
by the uniform convergence $\Oh\bigtriangleup \O \to 0$ stated in~\eqref{eq:uniform_conv_sets}
we have for $h$ sufficiently small that the nodal interpolant 
$\phi_h = \cI_h \phi \in \cS^1(\cT_h)$ satisfies $\phi_h \in H^1_0(\O_h)$.
Moreover, we deduce from nodal interpolation results and the uniform control 
on the shape of the elements that (after trivial extension to $Q$) we have
$\phi_h \to \phi$ in $H^1(Q)$. These properties lead to the relation
\[
\int_Q \nabla u \cdot \nabla \phi \,\dx 
= \lim_{h\to 0} \int_Q \nabla u_h \cdot \nabla \phi_h \,\dx 
= \lim_{h\to 0} \int_Q  f \phi_h \, \dx
= \int_Q f \phi \, \dx.
\]
Arguing \sbb{as in the proof of Proposition~\ref{prop:existence}}
it follows that the finite element solutions converge strongly in $H^1_0(Q)$.
\sbb{With the results from \cite[p.~213]{AttouchButtazzoMichaille2014}}
the conditions on $j$ yield that
\[
J(\O,u) \le \liminf_{h\to 0} J(\Oh,u_h).
\]
\sbb{It remains to show that the pair $(\O,u)$ is optimal. For this, we note
that for an optimal pair $(\O^*,u^*)$ the convex set $\O^*$ can be approximated 
by smooth convex domains contained in $\O^*$ and if also the right-hand side $f$ is 
regularized then the corresponding smooth solutions of the Poisson problems
converge strongly in $H^1$ to $u^*$, see~\cite[p. 141]{Dobro06} for related details. 
Hence,~\cref{prop:consistency}
implies that the limit of the sequence $J(\Oh,u_h)$ is bounded
from above by the optimal continuous value.}
Since $(\O,u)$ is admissible we deduce that it solves~\eqref{eq:prob}. 
\end{proof}

\subsection{Deformed triangulations}
The discrete formulation~\eqref{eq:prob_2_discr}
is of limited practical interest. Under a regularity assumption on a solution 
of the continuous formulation~\eqref{eq:prob} we devise a convergent
discrete scheme that admits a solution. The admissible discrete domains are 
obtained from certain discrete deformations of a given convex reference domain $\hO$:

\begin{equation}
	\label{eq:prob_2b_discr}
	\tag{$\textbf{P}_h'$}
	\begin{aligned}
		\text{Minimize}\quad & \int_\Oh j(x, u_h(x), \nabla u_h(x)) \,\dx \\
		\text{w.r.t.}\quad & \Phi_h \in \cS^1(\hcT_h)^d, \, 
                      \cT_h \text{ triangulation of } \Oh, \, u_h \in \cS^1_0(\cT_h) \\
		\text{s.t.}\quad
                & \|D\Phi_h\|_{L^\infty(\hO)} + \|[D\Phi_h]^{-1} \|_{L^\infty(\hO)} \le c_0 \\
                & \Oh = \Phi_h(\hO) \subset Q \text{ convex and open}, \, \cT_h = \Phi_h(\hcT_h) \\
		& -\Delta_h u_h = f_h \text{ in }\Oh, u_h = 0 \text{ on }\p\Oh 
	\end{aligned}
\end{equation}
Here, $\hcT_h$ is a fixed regular triangulation of the convex reference domain $\hO$. 
The notation $\Phi_h(\hcT_h)$ represents the triangulation $\cT_h$
of $\Oh$ that is obtained by applying $\Phi_h$ to the elements in $\hcT_h$. 

\begin{proposition}
Assume that $(\O,u)$ is a solution for~\eqref{eq:prob} such that $\O = \Phi(\hO)$ 
with an injective deformation $\Phi\in W^{1,\infty}(\hO;\R^d)$, satisfying 
\[
\|D\Phi\|_{L^\infty(\hO)} + \|[D\Phi]^{-1}\|_{L^\infty(\hO)} \le c_0'
\]
and $\Phi|_T \in W^{2,\infty}(T;\R^d)$ with $\|D^2\Phi\|_{L^\infty(T)}\le c_0''$ 
for all $T\in \hcT_h$.  For $h$ sufficiently small and
$c_0$ sufficiently large we have the following results: \\
(i) There exists a discrete solution $(\Oh,u_h)$ for~\eqref{eq:prob_2b_discr}. \\
(ii) If $d=2$ there exists an admissible pair $(\Oh,u_h)$ obeying
the bounds of \cref{prop:consistency}. \\
(iii) If $d=2$ and $(\O_h,u_h)_{h>0}$ is a sequence of discrete solutions then every
accumulation point for $h\to 0$ solves the continuous formulation.
\end{proposition}

\begin{proof}
(i) The existence of a solution $(\Phi_h,\O_h,u_h)$ follows from continuity 
properties of the objective function and the boundedness of admissible elements. \\
(ii) We define $\Phi_h = \hcI_h \Phi$ as the nodal interpolant of $\Phi$ on 
$\hO$. As $\|D\Phi_h -D\Phi\|_{L^\infty(T)} \le c h \|D^2\Phi\|_{L^\infty(T)}$ for all
$T\in \hcT_h$ we find that $\Phi_h$ is admissible in~\eqref{eq:prob_2b_discr}. 
Then, $\O_h = \Phi_h (\hO)$ is a polygonal domain whose boundary interpolates 
the boundary of $\O$ so that it is convex (in the considered two-dimensional situation).
The finite element solution $u_h$ on $\O_h$ then approximates the exact
solution $u\in H^1_0(\O)\cap H^2(\O)$ of the Poisson problem on the convex domain~$\O$. \\
(iii) For a sequence of discrete solutions it follows as in \cref{prop:stability}
that accumulation points $(\O,u)$ are admissible in~\eqref{eq:prob}. By the
approximation result~(ii) it follows that these are optimal. 
\end{proof}

\begin{remark}
As above, an interpolating polygonal domain of a convex domain is in general not
convex if $d=3$ so that the proposition cannot be directly generalized to that setting.
\end{remark}

The elementwise regularity assumption on a solution $\Phi$ cannot be avoided in general. 

\begin{example}[Non-invertibility of the interpolant]
Let $\Phi:B_{1,\pi/2}(0) \to B_{1,\pi}(0)$ the mapping that maps the quarter-disk with 
radius~1 to the half-disk with radius~1 by doubling the angle of every point in its polar 
coordinates, i.e., $\Phi(r,\phi) = (r,2\phi)$ for $0<r<1$ and $0<\phi<\pi/2$.
Then $\Phi$ and $\Phi^{-1}$ are Lipschitz continuous. For every $0<h\le 1$ 
the vertices of the triangle $T = \operatorname{conv}\{(0,0),(h,0),(0,h)\}$ are mapped onto the line segment 
$[-h,h]\times \{0\}$. Hence, the interpolant $\Phi_h = \cI_h \Phi$ cannot be a diffeomorphism
on $T$. 
\end{example}

\section{Numerical realization}
\label{sec:numerics}

In this section, we describe a possibility to solve
a slight variation concerning the treatment of the bounds on the diffeomorphism $\Phi_h$
(see \cref{subsec:restricted_gradient} for details) of 
the discretization~\eqref{eq:prob_2b_discr}
of \eqref{eq:prob}.
For this, let $\Omega_h$ be a convex polygon together with a 
regular triangulation $\cT_h$. We introduce the discrete shape functional $J_h$ via
\begin{equation}
	\label{eq:discrete_shape_functional}
	J_h(\Omega_h)
	:=
	\int_{\Omega_h} j(x, u_h(x), \nabla u_h(x)) \, \dx,
\end{equation}
where $u_h \in \cS^1_0(\cT_h)$ is the solution of the
discretized state equation
\begin{equation*}
	\int_\Oh \nabla u_h \cdot \nabla v_h \,\dx = \int_\Oh f \, v_h \, \dx
	\qquad\forall v_h \in \cS^1_0(\cT_h).
\end{equation*}
In fact, $J_h$ depends not only on $\Omega_h$
but also on the underlying triangulation $\cT_h$. However, this is dependence is 
not explicitly mentioned for ease of the presentation.

We are going to optimize this domain
by moving the vertices in the triangulation $\cT_h$.
We will see that this is consistent with a discrete version
of the perturbation of identity.
Indeed, if
$V_h \in \cS^1(\cT_h)^d$
is a piecewise linear deformation field,
then
$T_t := I + t \, V_h$
describes a piecewise linear perturbation
leading to the deformed triangulation
$T_t(\cT_h)$,
in which the position of each vertex $x_i$ is changed to
$x_i + t \, V_h(x_i)$.
Note that $V_h(x_i)$ are precisely the degrees of freedom of the
finite element function $V_h$.
Moreover, this deformation has the important property
\begin{equation*}
	v_h \circ T_t^{-1} \in \cS^1_0(T_t(\cT_h))
	\quad\Longleftrightarrow\quad
	v_h \in \cS^1_0(\cT_h)
\end{equation*}
for all functions $v_h : \Omega \to \R$.

Due to this property, it is possible to derive a
shape derivative for the discrete functional
$J_h$
along the same lines as used in \cref{sec:prelim}.
This leads to the expression
\begin{equation}
	\label{eq:discrete_shape_derivative}
	\begin{aligned}
	J_h'(\Omega_h; V_h)
	=
	\int_\Oh
	&j_x(\cdot) \cdot V_h
	-
	j_v(\cdot) \cdot DV_h^\top \nabla u_h
	+
	\nabla p_h^\top \bigl[-DV_h - DV_h^\top + \div(V_h) \,I\bigr]\nabla u_h
	\\
	&\qquad
	-
	\div(f \, V_h) \, p_h
	+
	j(\cdot) \, \div(V_h)
	\,\dx
	\end{aligned}
\end{equation}
for all $V_h \in \cS^1(\cT_h)^d$,
where $p_h \in \cS^1_0(\cT_h)$ is the discrete adjoint state
which solves
\begin{equation}
	\label{eq:discrete_adjoint_equation}
	\int_\Oh \nabla p_h\cdot \nabla w_h \, \dx
	=
	\int_\Oh j_u(\cdot) \, w_h + j_v(\cdot) \cdot \nabla w_h \, \dx
	\qquad\forall w_h \in \cS^1_0(\cT_h).
\end{equation}
Analogously to \cref{sec:prelim},
we used $(\cdot)$ to abbreviate the argument $(x, u_h(x), \nabla u_h(x))$.

In order to formulate an implementable algorithm,
we have to compute a ``good'' representing deformation field $V_h$
from the shape derivative $J_h'(\Omega_h;\cdot)$.
Consequently, this deformation field is used to update
the domain $\Omega_h$
via $(I + t\,V_h)(\Omega_h)$,
where $t$ is a suitable step size.
For the calculation of $V_h$,
three important points are to be considered
and these will be described in the next three sections:
\begin{enumerate}
	\item
		The shape derivative $J_h'(\Omega_h;\cdot)$
		is an element of the dual space of $\cS^1(\cT_h)^d$
		and has to be represented by 
		an element of $\cS^1(\cT_h)^d$ in an appropriate inner product.
		This will be considered in \cref{subsec:shape_gradient}.
	\item
		At some point in our algorithm,
		we have to respect the convexity constraint
		which is posed in the problem under consideration.
		We will use deformation fields
		which are (in a certain sense)
		first-order feasible, see \cref{subsec:feasible_def_fields}.
	\item
		We have seen in \cref{sec:prelim},
		that (in certain situations), $J'(\Omega;V)$
		only depends on the normal trace of $V$.
		This is no longer the case for $J'_h(\Oh;V_h)$,
		since the functional $J_h(\Oh)$
		also depends on the location of the inner nodes
		of the triangulation $\cT_h$.
		This problem and a possible resort
		are addressed in \cref{subsec:restricted_gradient}.
\end{enumerate}
With these preparations, we comment on a possible line-search strategy
(\cref{subsec:line_search}) and
state an implementable algorithm (\cref{subsec:algorithm}).

\subsection{Computation of a shape gradient}
\label{subsec:shape_gradient}
There are many possibilities to compute a deformation field $V_h$
from the linear functional $J_h'(\Oh;\cdot)$.
We follow the approach proposed in \cite[Section~3]{SchulzSiebenbornWelker2016}.
To this end,
we introduce the elasticity bilinear form
\begin{equation*}
	\cE_h(V_h, W_h)
	:=
	\int_\Oh
	2 \, \mu \, \bvarepsilon(V_h) \dprod \bvarepsilon(W_h)
	+
	\lambda \, \trace(\bvarepsilon(V_h)) \, \trace(\bvarepsilon(W_h))
	+
	\delta \, V_h \cdot W_h
	\,
	\dx
\end{equation*}
for $V_h, W_h \in \cS^1(\cT_h)^d$.
Here,
$\bvarepsilon(V_h) := (DV_h + DV_h^\top)/2$
is the linearized strain tensor.
Moreover,
$\mu, \lambda > 0$ are the Lamé parameters
and $\delta > 0$
is a damping parameter
such that $\cE_h$ becomes coercive on $\cS^1(\cT_h)^d$.

Now, one possibility to compute a deformation field $V_h$
is to solve
\begin{equation}
	\label{eq:discrete_riesz_representation_1}
	\cE_h(V_h, W_h)
	=
	-J'_h(\Omega_h; W_h)
	\qquad\forall W_h \in \cS^1(\cT_h)^d.
\end{equation}
Note that this is equivalent to solving the following minimization problem:
\begin{equation}
	\label{eq:discrete_riesz_representation_2}
	\begin{aligned}
		\text{Minimize}\quad& 
		\frac12\cE_h(V_h, V_h)
		+
		J'_h(\Omega_h; V_h)
		\\
		\text{w.r.t.}\quad&
		V_h \in \cS^1(\cT_h)^d
		.
	\end{aligned}
\end{equation}

\subsection{Feasible deformation fields}
\label{subsec:feasible_def_fields}
Using the deformation field $V_h$ from \eqref{eq:discrete_riesz_representation_1}
for deforming the domain $\Omega_h$ could lead to non-convex domains.
We incorporate the convexity constraint in such a way that the deformation field
respects the convexity constraint to first order.
We focus on the case of dimension $d = 2$ and briefly outline the case $d = 3$.

Let  $\Omega_h \subset \R^2$ be a simply connected polygon
and let $N$ be the number of boundary vertices of $\Omega_h$
with coordinates $x^{(i)} \in \R^2$, $i = 1,\ldots, N$,
in counterclockwise order. It is easily seen that $\Omega_h$ is convex
if and only if all the interior angles are less than or equal to $\pi$.
By using the cross product, this, in turn, is equivalent to
\begin{equation}
	\label{eq:discrete_convexity_2d}
	C_i(X)
	:=
	\bigl(x^{(i-1)}_1 - x^{(i)}_1\bigr)
	\,
	\bigl(x^{(i+1)}_2 - x^{(i)}_2\bigr)
	-
	\bigl(x^{(i-1)}_2 - x^{(i)}_2\bigr)
	\,
	\bigl(x^{(i+1)}_1 - x^{(i)}_1\bigr)
	\le
	0,
\end{equation}
for $i=1,\dots,N$, 
where we used the conventions $x^{(0)} = x^{(N)}$ and $x^{(N+1)} = x^{(1)}$.
Moreover, the argument $X$ represents the vector $(x^{(1)},\ldots,x^{(N)})$.
Likewise, the coordinates of the vertices of the perturbed domain
are given by $x_i + t^0 \, V_h(x_i)$,
where $t^0$ is an initial step size.
Using these functions $C_i$, the convexity of the deformed domain
$(I + t^0\,V_h)(\Omega_h)$ is equivalent to
$C_i(X + t^0 \, V_h(X)) \le 0$ for all $i = 1,\ldots,N$.
We are going to use a first-order expansion of this
quadratic constraint and obtain
\begin{equation*}
	C_i(X + t^0 \, V_h(X))
	\approx
	C_i(X) + t^0 \, D C_i(X) \, V_h(X)
	\stackrel{!}{\le}0,
	\quad \forall i = 1,\ldots,N.
\end{equation*}
Therefore, we replace \eqref{eq:discrete_riesz_representation_2} by
the constrained problem:
\begin{equation}
	\label{eq:discrete_riesz_representation_3}
	\begin{aligned}
		\text{Minimize}\quad& 
		\frac12\cE_h(V_h, V_h)
		+
		J'_h(\Omega_h; V_h)
		\\
		\text{w.r.t.}\quad&
		V_h \in \cS^1(\cT_h)^d,
		\\
		\text{s.t.}\quad&
		C_i(X) + t^0 \, D C_i(X) \, V_h(X)
		\le 0,
		\quad \forall i = 1,\ldots,N
	\end{aligned}
\end{equation}
We remark that this is a convex quadratic program (QP).

We briefly comment on the three-dimensional situation.
Similar to the two-dimensional situation,
we consider a $2$-connected polyhedron $\Omega_h$.
Now, one can check that $\Omega_h$ is convex,
if each outer edge of $\Omega_h$ is convex in the sense that the
dihedral angle between the two adjacent faces
is less than or equal to $\pi$.
We show that this can be written as a system of polynomial inequalities
which is cubic w.r.t.\ the coordinates of the vertices of $\cT_h$.
To this end, we take an arbitrary outer edge
with vertices $x^{(i)}$ and $x^{(j)}$.
Relative to the vector from $x^{(i)}$ to $x^{(j)}$
we denote the third vertex of the left and right triangle
by $x^{(l)}$ and $x^{(r)}$, respectively.
The convexity of the edge can be characterized by the non-negativity of the signed volume
of the parallelepiped spanned by the vectors
$x^{(l)} - x^{(i)}$, $x^{(1)} - x^{(i)}$, $x^{(r)} - x^{(i)}$, i.e.,
\begin{equation*}
	\bigl( x^{(l)} - x^{(i)} \bigr)
	\cdot
	\Bigl(
	\bigl( x^{(j)} - x^{(i)} \bigr)
	\times
	\bigl( x^{(r)} - x^{(i)} \bigr)
	\Bigr)
	\ge
	0
\end{equation*}
Now, a QP similar to \eqref{eq:discrete_riesz_representation_3}
can be constructed analogously to the two-dimensional case.

\subsection{Avoiding spurious interior deformations}
\label{subsec:restricted_gradient}
We have already mentioned that the value of $J_h(\Omega_h)$
also depends on the positions of the interior nodes of the triangulation $\cT_h$,
since the discrete state $u_h$ depends on all the nodes of $\cT_h$.
Therefore,
using $V_h$ governed by \eqref{eq:discrete_riesz_representation_2}
or \eqref{eq:discrete_riesz_representation_3}
would result also in an optimization of the nodes of the triangulation.
In a more extreme case,
we could even fix all the boundary nodes of the triangulation $\cT_h$
in order to optimize the location of the interior nodes
(this would result in zero Dirichlet boundary conditions for $V_h$
in \eqref{eq:discrete_riesz_representation_2} or \eqref{eq:discrete_riesz_representation_3}).
This may lead to degenerate triangulations. 

	In the discrete problem \eqref{eq:prob_2b_discr},
	this degeneracy of the triangulations
	was avoided by the bounds on $D\phi_h$
	and $[D\phi_h]^{-1}$.
	In the numerical implementation,
	the realization of these constraints is rather cumbersome
	and it is also not clear how the constant $c_0$
	should be chosen.
	Therefore,
	we use a
possible resort which is proposed in the recent preprint
\cite{EtlingHerzogLoayzaWachsmuth2018}.
Therein,
the authors suggest to restrict the set of admissible
deformation fields in \eqref{eq:discrete_riesz_representation_2}.
Motivated by the continuous situation
in which $J'(\Omega; V)$ only acts on the normal trace of $V$
(due to Hadamard's structure theorem),
it is reasonable to only consider those deformation fields $V_h$
which result from a normal force.
To this end, we introduce the operator
$N_h : \cS^1(\partial\cT_h) \to (\cS^1(\cT_h)^d)^\star$
by requiring that for all $F_h \in \cS^1(\partial\cT_h)$ and $V_h \in \cS^1(\cT_h)^d$
we have
\begin{equation*}
	\dual{N_h \, F_h}{V_h}
	:=
	\int_{\partial\Oh} F_h \, (V_h \cdot n) \, \ds,
\end{equation*}
where $\cS^1(\partial\cT_h)$
are the piecewise linear and continuous functions
on the boundary $\partial\Oh$.
Now, a deformation field $V_h$ results from a normal 
force $F_h \in \cS^1(\partial\cT_h)$, if
\begin{equation*}
	\cE_h(V_h, W_h) = 
	\dual{N_h \, F_h}{W_h}
	\qquad\forall W_h \in \cS^1(\cT_h)^d.
\end{equation*}
Following this approach,
we replace \eqref{eq:discrete_riesz_representation_3} by the following minimization 
problem: 
\begin{equation}
	\label{eq:discrete_riesz_representation_4}
	\begin{aligned}
		\text{Minimize}\quad& 
		\frac12\cE_h(V_h, V_h)
		+
		J'_h(\Omega_h; V_h)
		\\
		\text{w.r.t}\quad&
		V_h \in \cS^1(\cT_h)^d,
		F_h \in \cS^1(\partial\cT_h),
		\\
		\text{s.t.}\quad&
		C_i(X) + t^0 \, D C_i(X) \, V_h
		\le 0,
		\quad \forall i = 1,\ldots,N
		,
		\\
	        &
		\cE_h(V_h, W_h) = 
		\dual{N_h \, F_h}{W_h}
		\quad \forall W_h \in \cS^1(\cT_h)^d
	\end{aligned}
\end{equation}
Again, problem \eqref{eq:discrete_riesz_representation_4} is a convex QP. The
problem is feasible if, e.g., the current domain $\O_h$ is convex. Otherwise, 
$\O_h$ can be replaced by its convex hull. 

\subsection{Armijo-like line search with merit function}
\label{subsec:line_search}
As explained above, we use the solution
$V_h$ of the convex QP \eqref{eq:discrete_riesz_representation_4}
as a search direction using the perturbation of identity approach.
That is,
the next iterate is given by
$\Omega_{h,t} := (I + t \, V_h)(\Omega_h)$,
where $t > 0$ is a suitable step size.
Since we consider a constrained problem,
we use a merit function for the determination
of suitable step sizes.
To this end,
we consider the merit function
\begin{equation*}
	\varphi(t)
	:=
	J_h(\Omega_{h,t})
	+
	M \, \sum_{i = 1}^N \bigl[ C_i(X + t\,V_h(X)) \bigr]^+
\end{equation*}
for some suitable parameter $M > 0$.
It is easy to check that the (directional) derivative 
$\varphi'(0) = \lim_{t \searrow 0} (\varphi(t)-\varphi(0))/t$ is given by
\begin{equation*}
	\varphi'(0)
	=
	J_h'(\Omega_h; V_h)
	+
	M \, \sum_{i:C_i(X) = 0} \bigl[ DC_i(X) \, V_h(X) \bigr]^+
	+
	M \, \sum_{i:C_i(X) > 0} DC_i(X) \, V_h(X)
	.
\end{equation*}
Moreover, if $V_h$ is a feasible point of \eqref{eq:discrete_riesz_representation_4},
the middle term vanishes due to the constraint in \eqref{eq:discrete_riesz_representation_4}.
Thus,
\begin{equation}
	\label{eq:derivative_merit_function}
	\varphi'(0)
	=
	J_h'(\Omega_h; V_h)
	+
	M \, \sum_{i:C_i(X) > 0} DC_i(X) \, V_h(X)
	.
\end{equation}
Next, we obtain the result that the solution of \eqref{eq:discrete_riesz_representation_4}
is a descent direction for the merit function $\varphi$,
if $M$ is chosen large enough.
Such a result is well known for the usage of merit functions in other optimization methods,
e.g., within the SQP method.
\begin{lemma}
	\label{lem:descent_merit_function}
	Let $V_h \ne 0$ be the solution of \eqref{eq:discrete_riesz_representation_4}.
	Then, $\varphi'(0) < 0$
	if
	$M \ge t_0 \, \sup_{i = 1,\ldots,N} \lambda_i$,
	where $\lambda \in \R^N$ is a Lagrange multiplier for $V_h$
	associated with the linearized convexity constraint in \eqref{eq:discrete_riesz_representation_4}.
\end{lemma}
Note that all the constraints in \eqref{eq:discrete_riesz_representation_4} are linear.
Therefore the existence of Lagrange multipliers at the solution $V_h$ follows
from standard arguments in optimization.
\begin{proof}
	We start by writing down the optimality conditions for \eqref{eq:discrete_riesz_representation_4}.
        Abbreviating throughout this proof the expression $DC_i(X)V_h(X)$ by $DC_i(X)V_h$, 
	the Lagrange function is given by
	\begin{equation*}
          \begin{split}
		L(V_h, F_h, \lambda, W_h)
		= & 
		\frac12 \cE_h(V_h,V_h)
		+
		J_h'(\Omega_h; V_h) \\
		& +
		\sum_{i = 1}^N \lambda_i \, \bigl( C_i(X) + t^0 \, DC_i(X) \, V_h \bigr)
		+
		\cE_h(V_h, W_h) - \dual{N_h \, F_h}{W_h}.
          \end{split}
	\end{equation*}
	By optimality of $V_h$,
	we obtain the existence of Lagrange multipliers
	$\lambda \in \R^n$, $W_h \in \cS^1(\cT_h)^d$
	such that
	\begin{align*}
		\cE_h(V_h,\delta V_h)
		+
		J_h'(\Omega_h; \delta V_h)
		+
		\sum_{i = 1}^N \lambda_i \, t^0 \, DC_i(X) \, \delta V_h
		+
		\cE_h(\delta V_h, W_h)
		&= 0,\\
		\dual{N_h \, \delta F_h}{W_h}
		&= 0,\\
		0 \le \lambda_i \quad\perp\quad C_i(X) + t^0 \, DC_i(X) \, V_h
		&\le 0,
	\end{align*}
	holds for all
	$\delta V_h \in \cS^1(\cT_h)^d$,
	$\delta F_h \in \cS^1(\partial\cT_h)$
	and $i = 1,\ldots, N$.
	Using $\delta F_h = F_h$ in the second equation
	together with the second constraint in \eqref{eq:discrete_riesz_representation_4},
	we find
	\begin{equation*}
		0
		=
		\dual{N_h \, F_h}{W_h}
		=
		\cE_h(V_h, W_h).
	\end{equation*}
	Together with
	$\delta V_h = V_h$ in the first equation of the optimality conditions,
	we obtain
	\begin{align*}
		0
		&>
		-\cE_h(V_h, V_h)
		=
		-\cE_h(V_h, V_h)
		-\cE_h(V_h, W_h)
		\\
		&=
		J_h'(\Omega_h; V_h)
		+
		\sum_{i = 1}^N \lambda_i \, t^0 \, DC_i(X) \, V_h.
	\end{align*}
	In order to obtain the claim,
	we distinguish two cases.

	Case~1, $C_i(X) \le 0$:
	If $\lambda_i > 0$, then the constraint
	$C_i(X) + t^0 \, DC_i(X)\,V_h \le 0$
	in \eqref{eq:discrete_riesz_representation_4}
	has to be active.
	This gives $DC_i(X)\,V_h \ge 0$.
	Thus, we obtain
	\begin{equation*}
		\lambda_i \, t^0 \, DC_i(X)\,V_h
		\ge
		0.
	\end{equation*}
	The same inequality is true if $\lambda_i = 0$.

	Case~2, $C_i(X) > 0$:
	The constraint in \eqref{eq:discrete_riesz_representation_4} implies
	$DC_i(X)\,V_h < 0$.
	Together with
	$\lambda_i \, t^0 \le M$,
	we obtain
	\begin{equation*}
		\lambda_i \, t^0 \, DC_i(X)\,V_h
		\ge
		M \, DC_i(X)\,V_h.
	\end{equation*}

	Now,
	the claim follows
	from these two cases and the representation \eqref{eq:derivative_merit_function}
	of $\varphi'(0)$.
\end{proof}

The overall line-search with backtracking is performed as follows.
We choose two parameters, $\sigma \in (0,1)$ and $\beta \in (0,1)$.
Then, we select the smallest non-negative integer $k$
such that
the Armijo condition
\begin{equation}
	\label{eq:armijo_backtracking}
	\varphi( t^0 \, \beta^k )
	\le
	\varphi( 0 )
	+
	\sigma \, t^0 \, \beta^k \, \varphi'(0)
\end{equation}
holds.
Moreover, we need to check that the mesh quality
is not affected too badly by the deformation $V_h$.
Therefore, we check that
\begin{equation}
	\label{eq:mesh_quality}
	\frac12 \le \det(I + t^0 \, \beta^k \, DV_h) \le 2
	,
	\qquad
	\lVert t^0 \, \beta^k \, DV_h \rVert \le 0.3
\end{equation}
are satisfied in the entire domain.
Note that this amount to checking three inequalities
per cell of the mesh.

\subsection{An implementable algorithm}
\label{subsec:algorithm}
Now, we are in the position
to state an implementable algorithm,
see \cref{alg:alg}.

\begin{algorithm}[htb]
	\SetAlgoLined
	\KwData{Initial domain $\Omega_h$ with triangulation $\cT_h$\\
		Initial step size $t_0 > 0$,
		convergence tolerance $\varepsilon_{\textup{tol}} > 0$,\\
		Parameters $\beta \in (0,1)$, $\beta_M > 1$, $\sigma \in (0,1)$, $M > 0$
	}
	\KwResult{Improved domain $\Omega_h$}
	% initialization\;
	\For{$i \leftarrow 1$ \KwTo $\infty$}{
		Set initial step size for iteration $i$: $t^0 \leftarrow t_{i-1} / \beta$\;
		Set up and solve the convex QP \eqref{eq:discrete_riesz_representation_4}\;
		\If{$\sqrt{\abs{J_h'(\Omega_h;V_h)}} \le \varepsilon_{\textup{tol}}$}{
			STOP, the current iterate $\Omega_h$ is almost stationary\;
		}
		\While{$\varphi'(0) \ge 0$, cf.\ \eqref{eq:derivative_merit_function}}{
			Increase the parameter $M$ of the merit function: $M \leftarrow M \, \beta_M$\;
		}
		$k \leftarrow 0$\;
		\While{\eqref{eq:armijo_backtracking} or \eqref{eq:mesh_quality} is violated}{
			$k \leftarrow k + 1$\;
		}
		$t_i \leftarrow t^0 \, \beta^k$\;
		Move the domain and triangulation according to $\Omega_h \leftarrow (I + t_i \, V_h)(\Omega_h)$,
		$\cT_h \leftarrow (I + t_i \, V_h)(\cT_h)$\;
	}
	\caption{Solving a discretized problem}
	\label{alg:alg}
\end{algorithm}

\section{Numerical examples}
%%fakesubsection: Intro
\label{sec:num_ex}
In this section, we present numerical examples that illustrate the performance of
the proposed algorithm and typical qualitative features of optimal convex shapes. 
We have implemented \cref{alg:alg} in Python.
For the finite-element discretization,
we utilized the FEniCS framework
\cite{AlnaesBlechtaEtAl2015,LoggWellsHake2012}.
The QP \eqref{eq:discrete_riesz_representation_4}
was solved using OSQP \cite{StellatoBanjacEtAl2017}.

Before we discuss the different examples,
we specify some of the required parameters:
\begin{equation*}
	\beta = \frac12,
	\qquad
	\sigma = \frac1{10},
	\qquad
	t_0 = 1,
	\qquad
	\beta_M = 10,
	\qquad
	M = 10^{-9},
	\qquad
	\varepsilon_{\textup{tol}}
	=
	10^{-6}.
\end{equation*}

\subsection{Example in 2 dimensions}
\label{subsec:ex1}
The first example shows that mesh deformations have to be carefully constructed
to avoid degenerate triangulations. We consider problem
\eqref{eq:prob} in two dimensions ($d = 2$), the objective is given by
\begin{equation*}
	\int_\Omega u \, \dx,
	\quad\text{i.e.,}\quad
	j(x, u, g)
	=
	u
\end{equation*}
and the right-hand side in the PDE is
\begin{equation*}
	f(x_1, x_2)
	=
	20 \, (x_1 + 0.4 - x_2^2)^2 + x_1^2 + x_2^2 - 1
	.
\end{equation*}
% Zero-level set in matlab:
% x = linspace(-1,1,1001); y = x; [X,Y] = meshgrid(x,y);
% Z = 20 *(X + .4 - Y.^2).^2 + X.^2 + Y.^2 - 1;
% surf(X,Y,Z)
% contour(X,Y,Z,[0,0])
The motivation for using this right-hand side $f$ is as follows.
Since we are going to minimize $\int_\Omega u \, \dx$
all points $x$ with $f(x) \le 0$
are favorable since they decrease $u$
(due to the maximum principle).
Moreover, in the absence of a convexity constraint the 
optimal domain would be a subset of the non-convex level
$Z_0 = \{x \in \R^2 \mid f(x) \le 0\}$.
Therefore, the shape optimization problem
has to find a compromise between convexity of $\Omega$
and covering all of $Z_0$.

We choose the unit circle as the initial domain and 
started with a rather coarse discretization,
see the upper left plot in \cref{fig:solution}.
We use \cref{alg:alg} to optimize this discrete domain.
Afterwards, we applied a mesh refinement.
This process was repeated to obtain
optimal domains for five different refinements,
see \cref{alg:alg}.
The overall runtime was roughly half an hour.
The solution of the QPs \eqref{eq:discrete_riesz_representation_4}
dominate the overall runtime.
On the last three levels of discretization,
one solution of the QP \eqref{eq:discrete_riesz_representation_4}
was taking
0.05s, 0.5s and 15s, respectively.
The number of iterations (i.e., number of solutions of \eqref{eq:discrete_riesz_representation_4}) on the five refinement levels are
$36$, $36$, $83$, $95$ and $130$, respectively.
\begin{figure}[p]
	\centering
	\includegraphics[width=.28\textwidth]{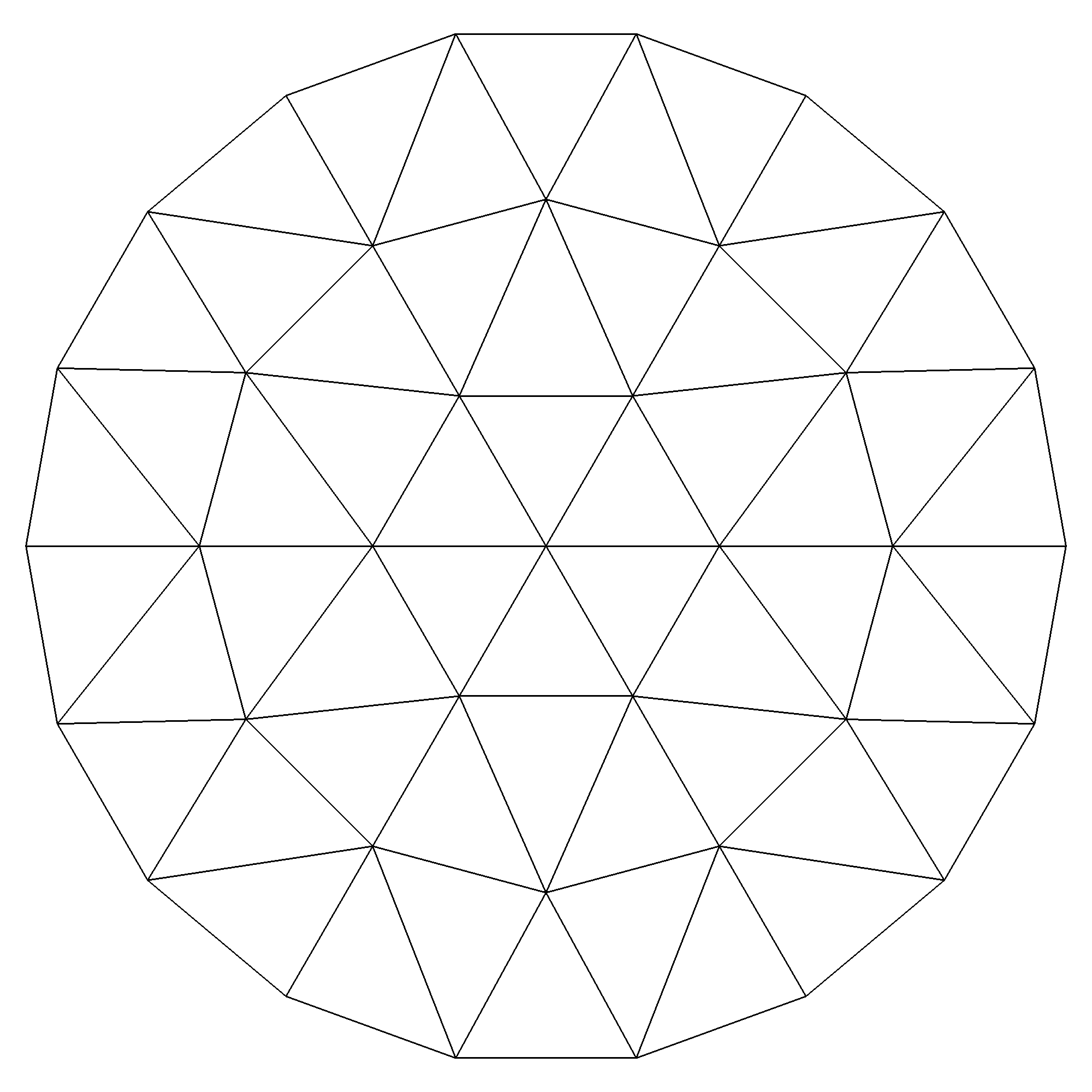}\hfill%
	\includegraphics[width=.28\textwidth]{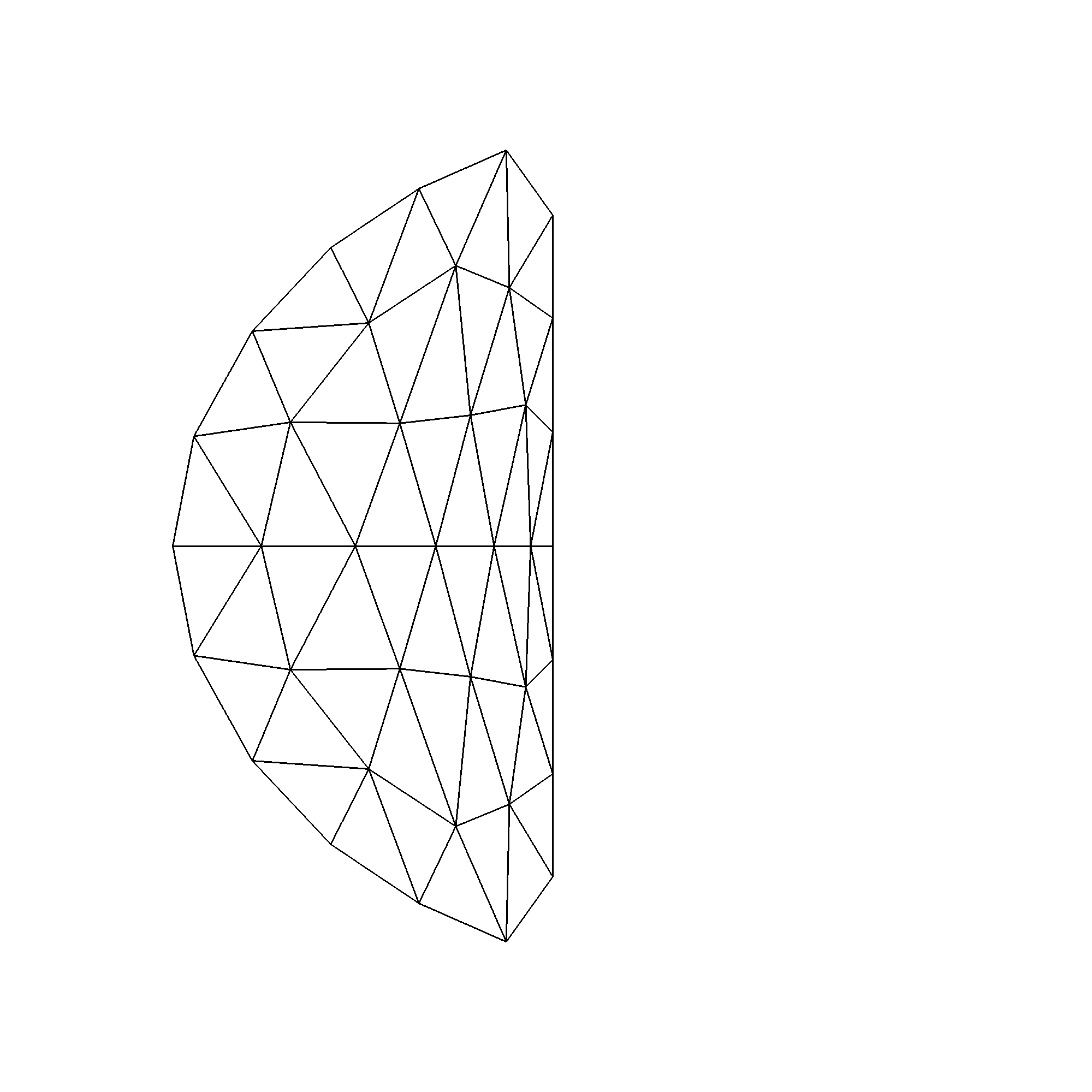}\hfill%
	\includegraphics[width=.28\textwidth]{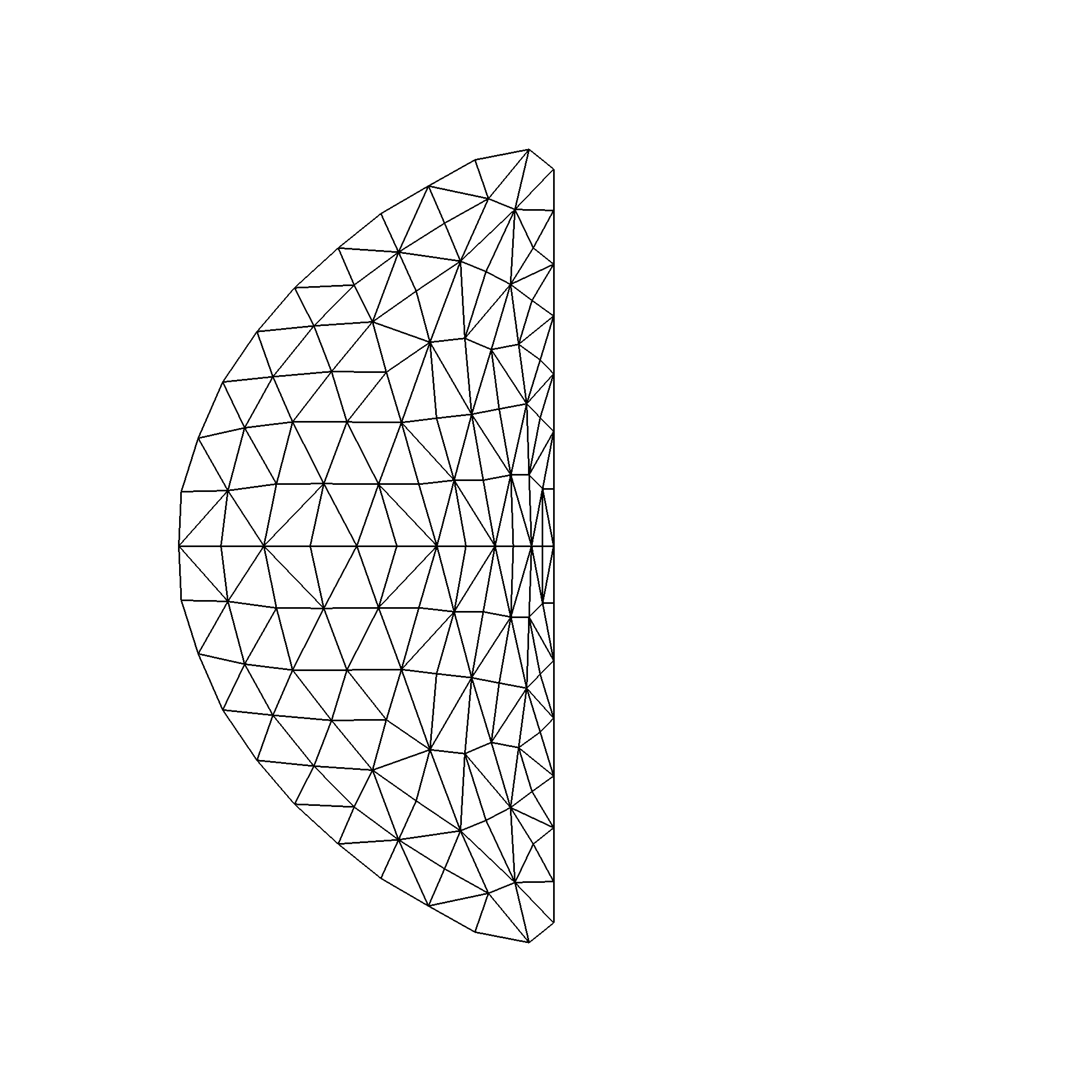}\\
	\includegraphics[width=.28\textwidth]{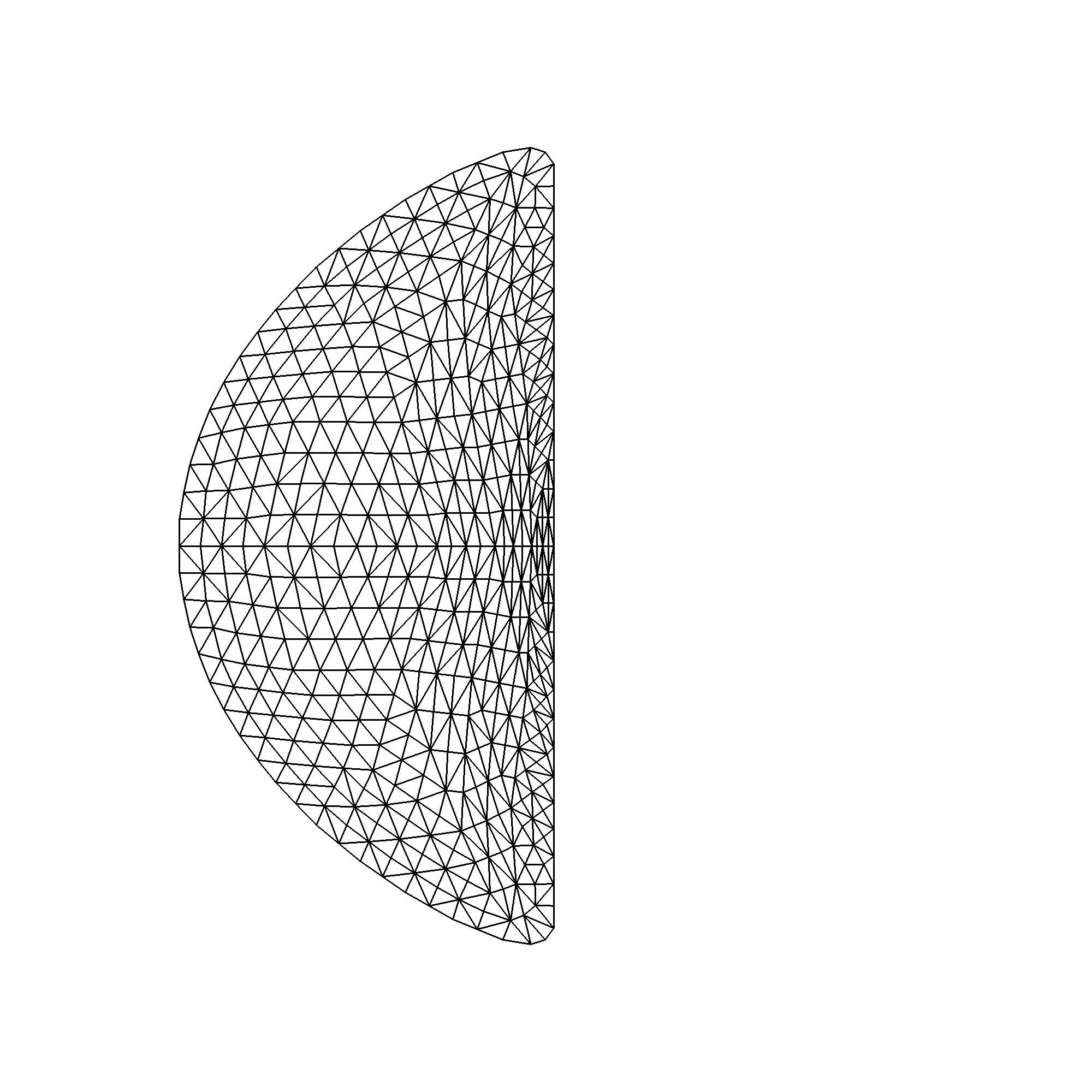}\hfill%
	\includegraphics[width=.28\textwidth]{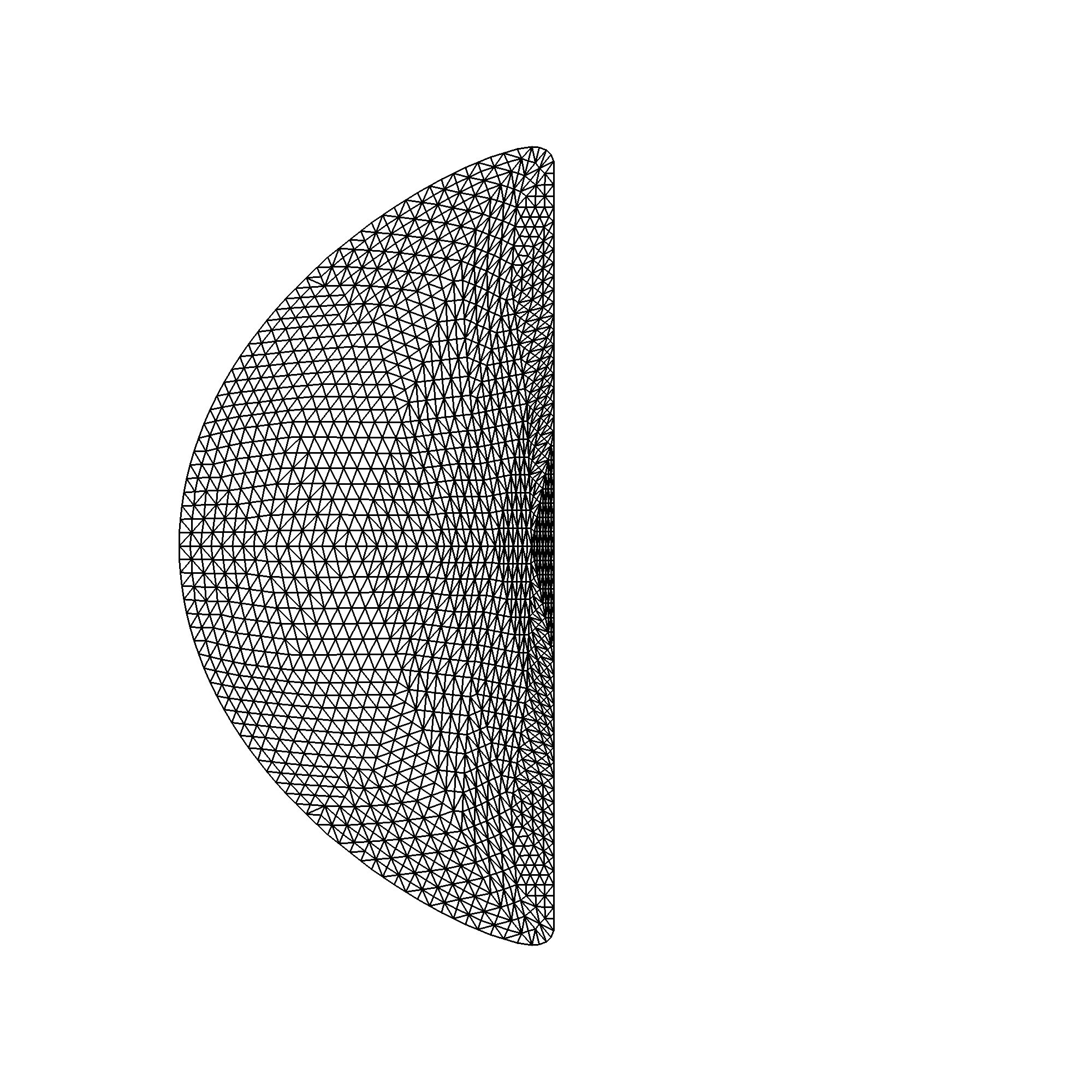}\hfill%
	\includegraphics[width=.28\textwidth]{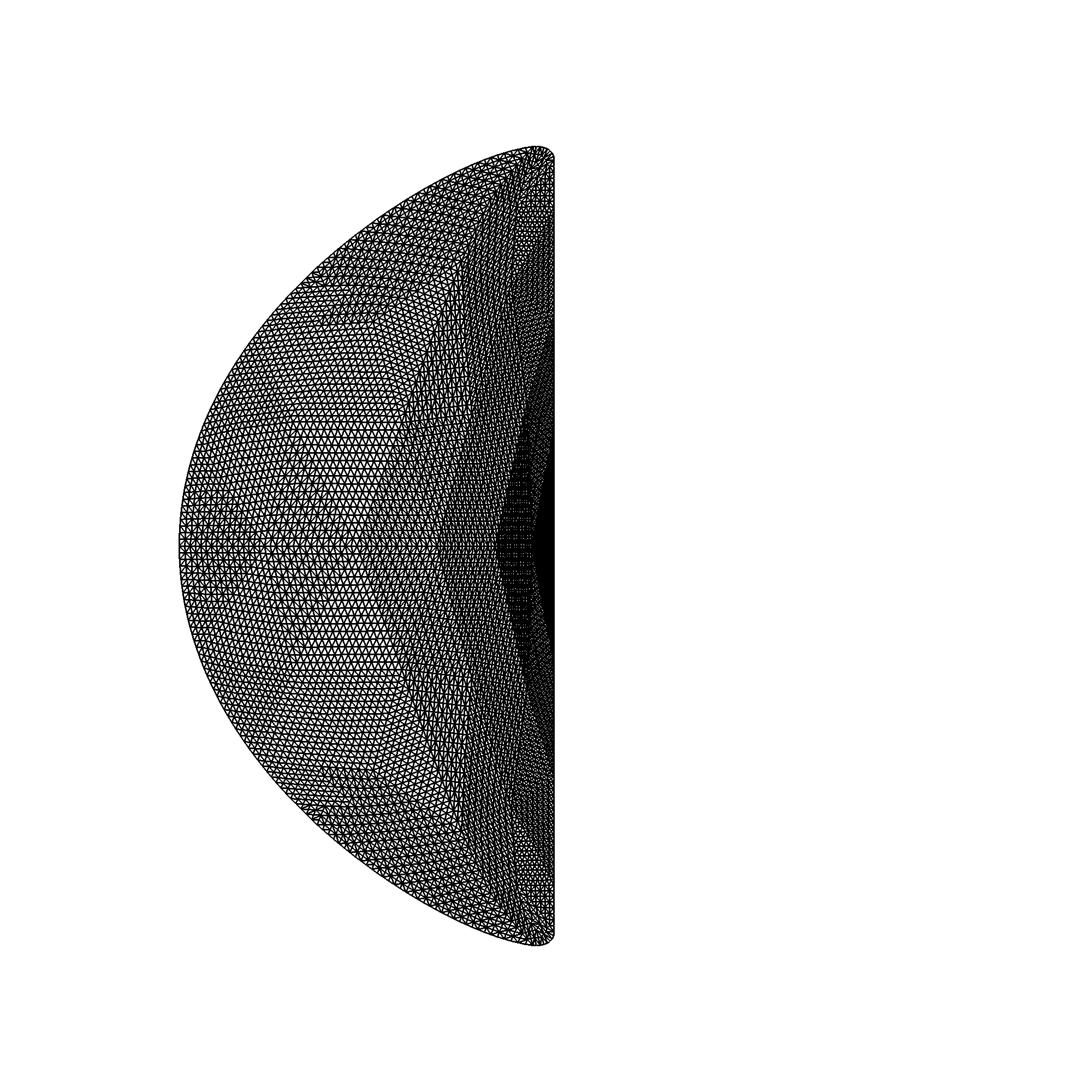}%
	\caption{Initial domain and coarse triangulation (upper left plot) and plots of optimal shapes 
         for refined triangulations for the example discussed in \cref{subsec:ex1}. The shapes of triangles 
         near the origin tend to deteriorate.}
	\label{fig:solution}
\end{figure}

\subsection{Example in 2 dimensions with non-smooth minimizer}
\label{subsec:ex2}
Our second example reveals that optimal convex shapes may have kinks and
that the $C^1$ regularity result of~\cite{Bucur2003} does not apply in the considered
framework. 
In this problem, we use the same data as in \cref{subsec:ex1}, but 
\begin{align*}
	f(x_1,x_2)
	=
	-\frac12 + \frac45 \, (x_1^2 + x_2^2)
	+
	2 \, &\sum_{i = 0}^{n-1}
	\exp\parens[\big]{-8\,( (x_1 - y_{1,i})^2 + (x_2 - y_{2,i})^2)}
	\\
	-
	&\sum_{i = 0}^{n-1}
	\exp\parens[\big]{-8\,( (x_1 - z_{1,i})^2 + (x_2 - z_{2,i})^2)}
\end{align*}
with $n = 5$ and
\begin{align*}
	y_{1,i} &= \sin( (i + 1/2) \, 2\,\pi / n),
	&
	z_{1,i} &= \frac65 \, \sin( i \, 2\,\pi / n),
	\\
	y_{2,i} &= \cos( (i + 1/2) \, 2\,\pi / n),
	&
	z_{2,i} &= \frac65 \, \cos( i \, 2\,\pi / n).
\end{align*}
% Zero-level set in matlab:
% n = 5
% x = linspace(-1,1,1001); y = x; [X,Y] = meshgrid(x,y);
% Z = -1/2 + 4/5 * (X.^2 + Y.^2);
% for i=0:4
% 	z1 = sin((i+1/2)*2*pi/n);
% 	z2 = cos((i+1/2)*2*pi/n);
% 	Z = Z + 2*exp(-8*( (X-z1).^2 + (Y-z2).^2));
% 	z1 = 6/5*sin((i+0/2)*2*pi/n);
% 	z2 = 6/5*cos((i+0/2)*2*pi/n);
% 	Z = Z - exp(-8*( (X-z1).^2 + (Y-z2).^2));
% end
% surf(X,Y,Z)
% contour(X,Y,Z,[0,0])
This function $f$ is designed to a have a
five-fold rotational symmetry.
In particular, the contribution of the second summation operator
attracts the optimal shape towards the points $(z_{1,i}, z_{2,i})$,
whereas the first summation operator
repels the optimal shape from the points $(y_{1,i}, y_{2,i})$.
To solve the discrete problem,
we performed the same steps as for the previous example.
The initial mesh and the optimal mesh of five subsequent
refinements can be seen in \cref{fig:nonsmooth}.
\begin{figure}[p]
	\centering
	\includegraphics[width=.25\textwidth]{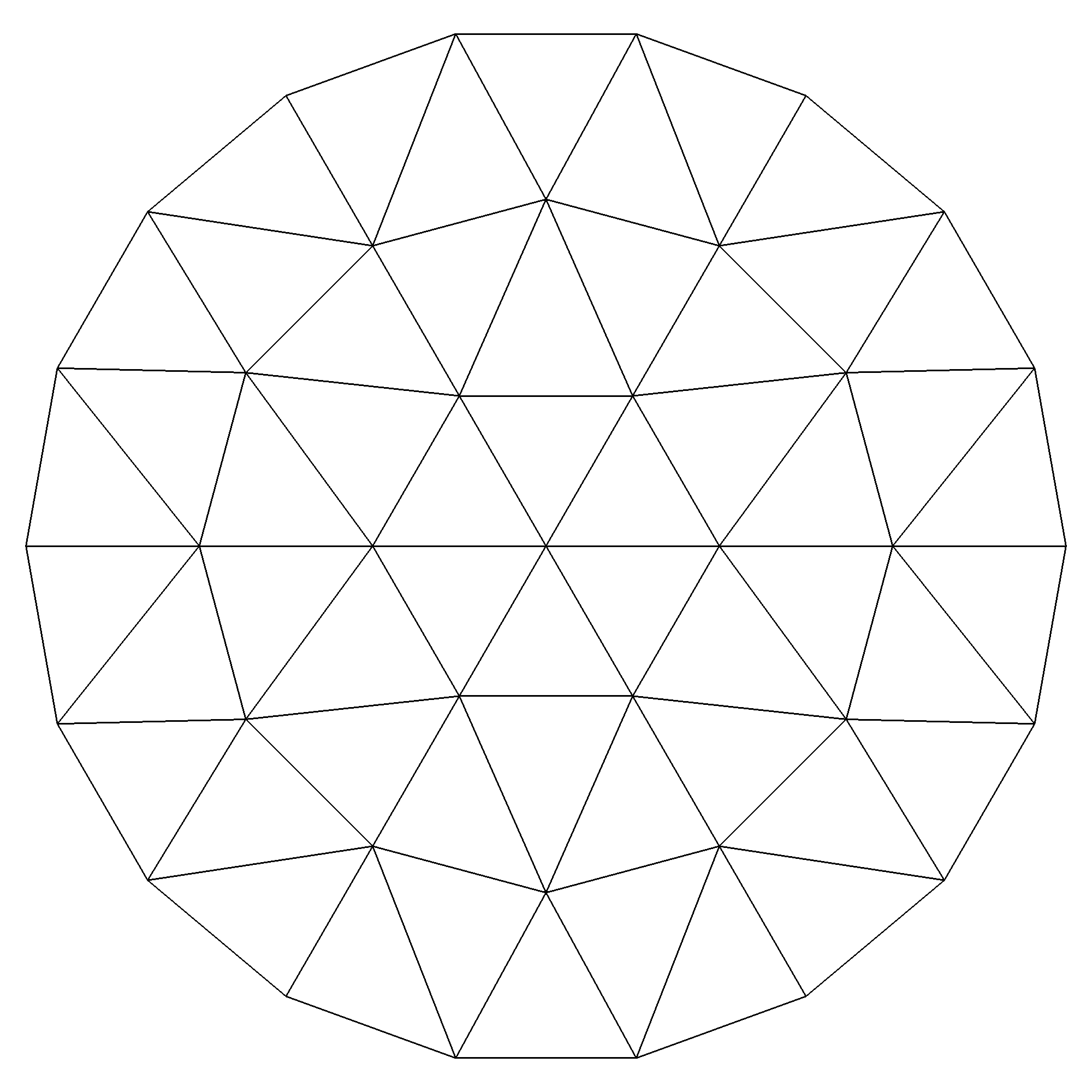}\hfill%
	\includegraphics[width=.28\textwidth]{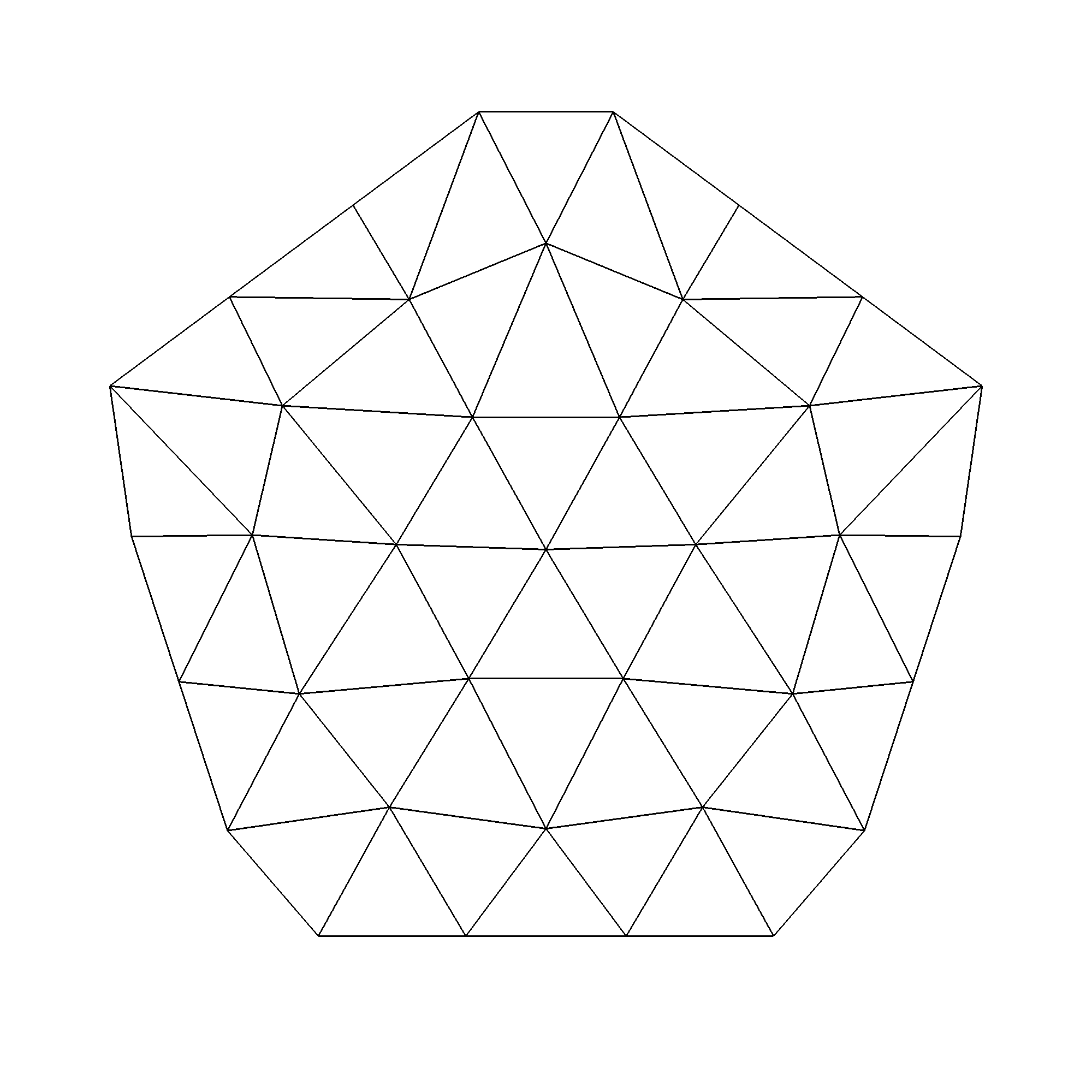}\hfill%
	\includegraphics[width=.28\textwidth]{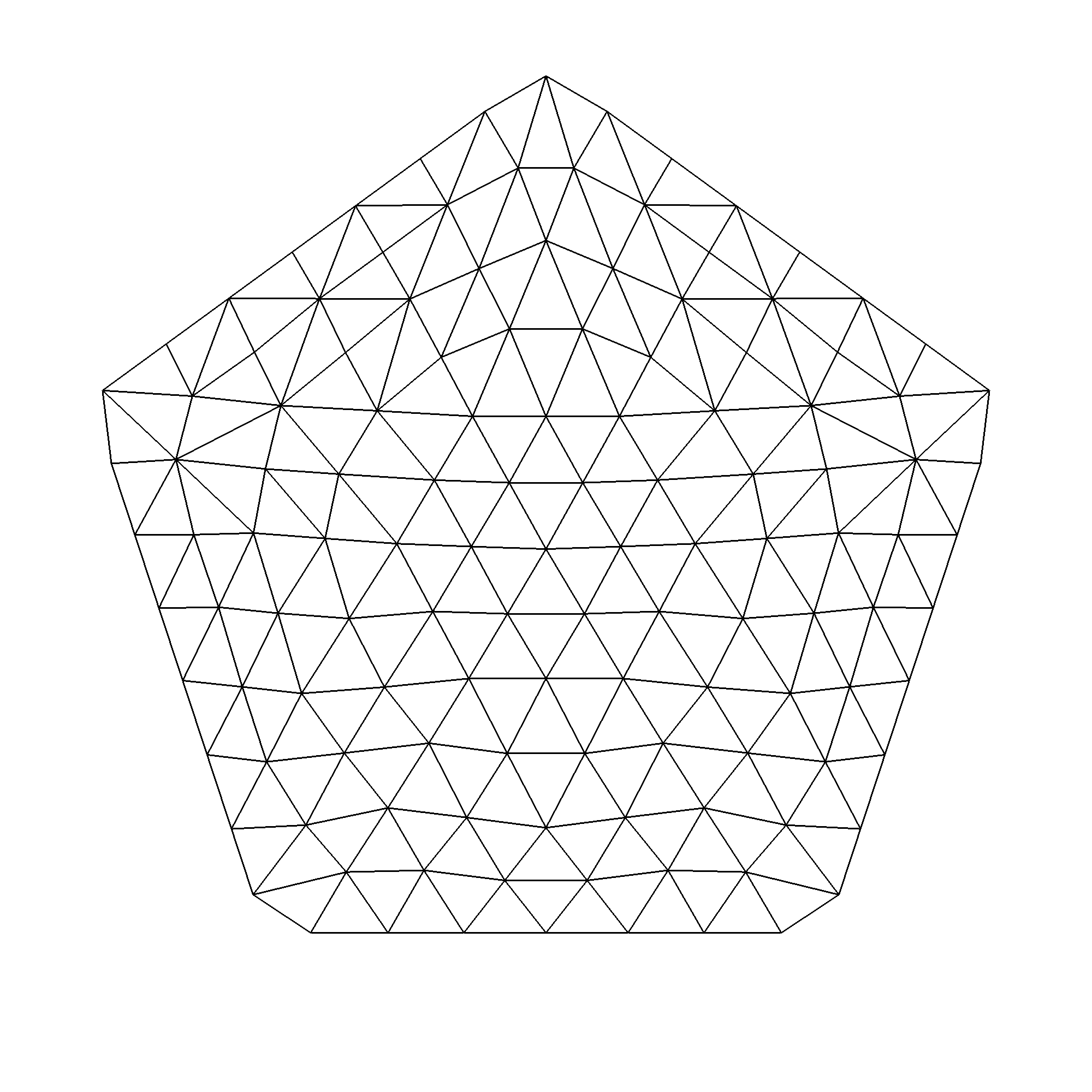}\\
	\includegraphics[width=.28\textwidth]{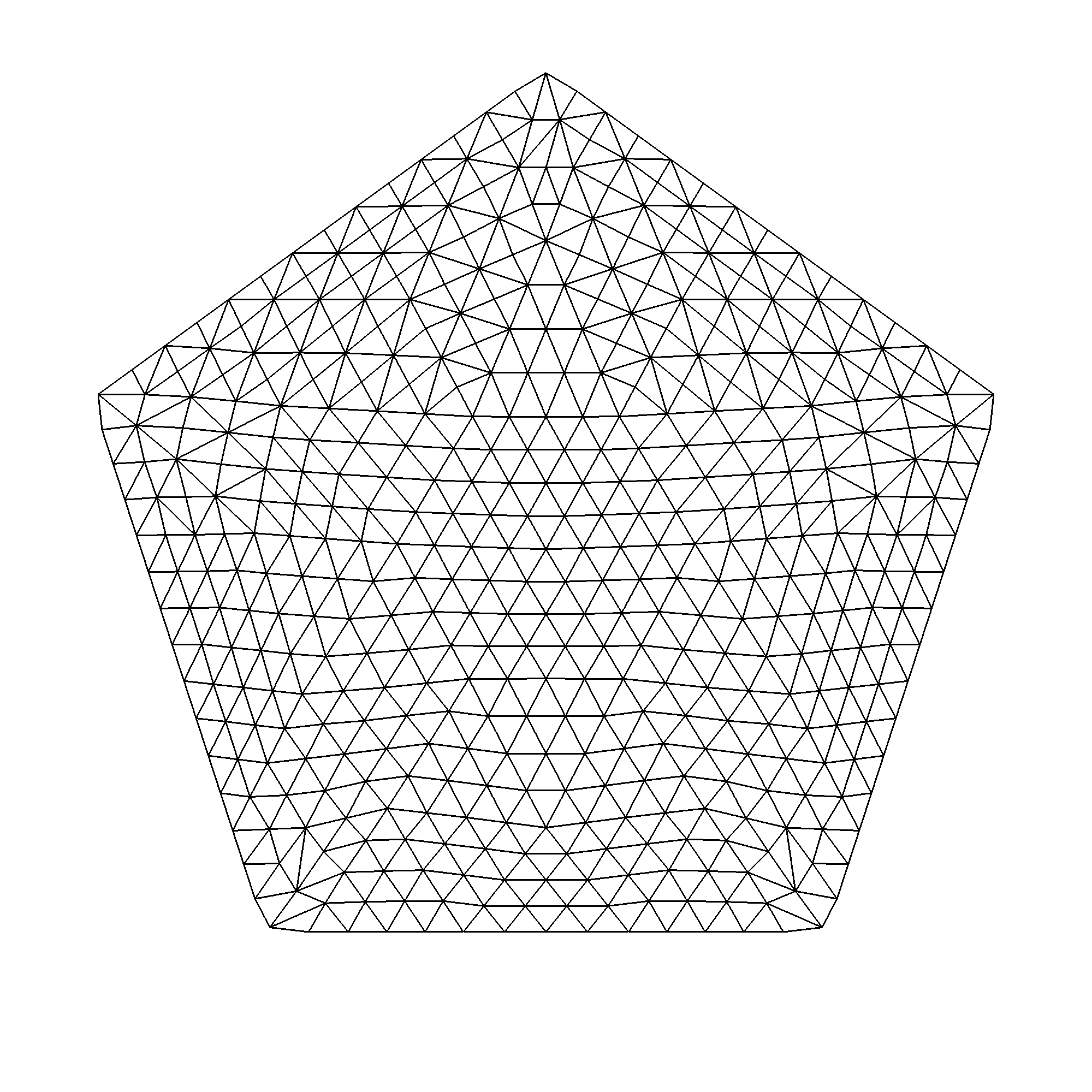}\hfill%
	\includegraphics[width=.28\textwidth]{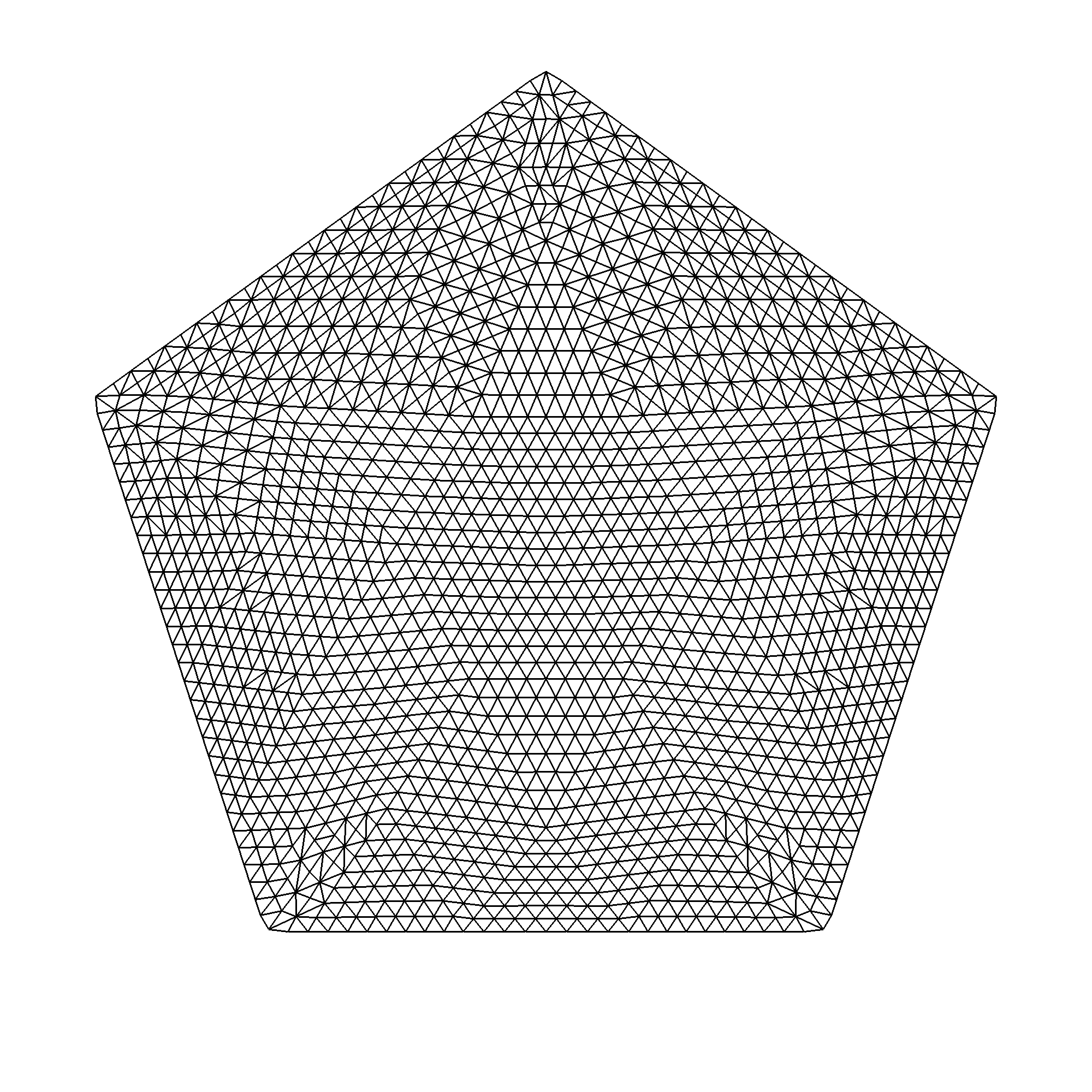}\hfill%
	\includegraphics[width=.28\textwidth]{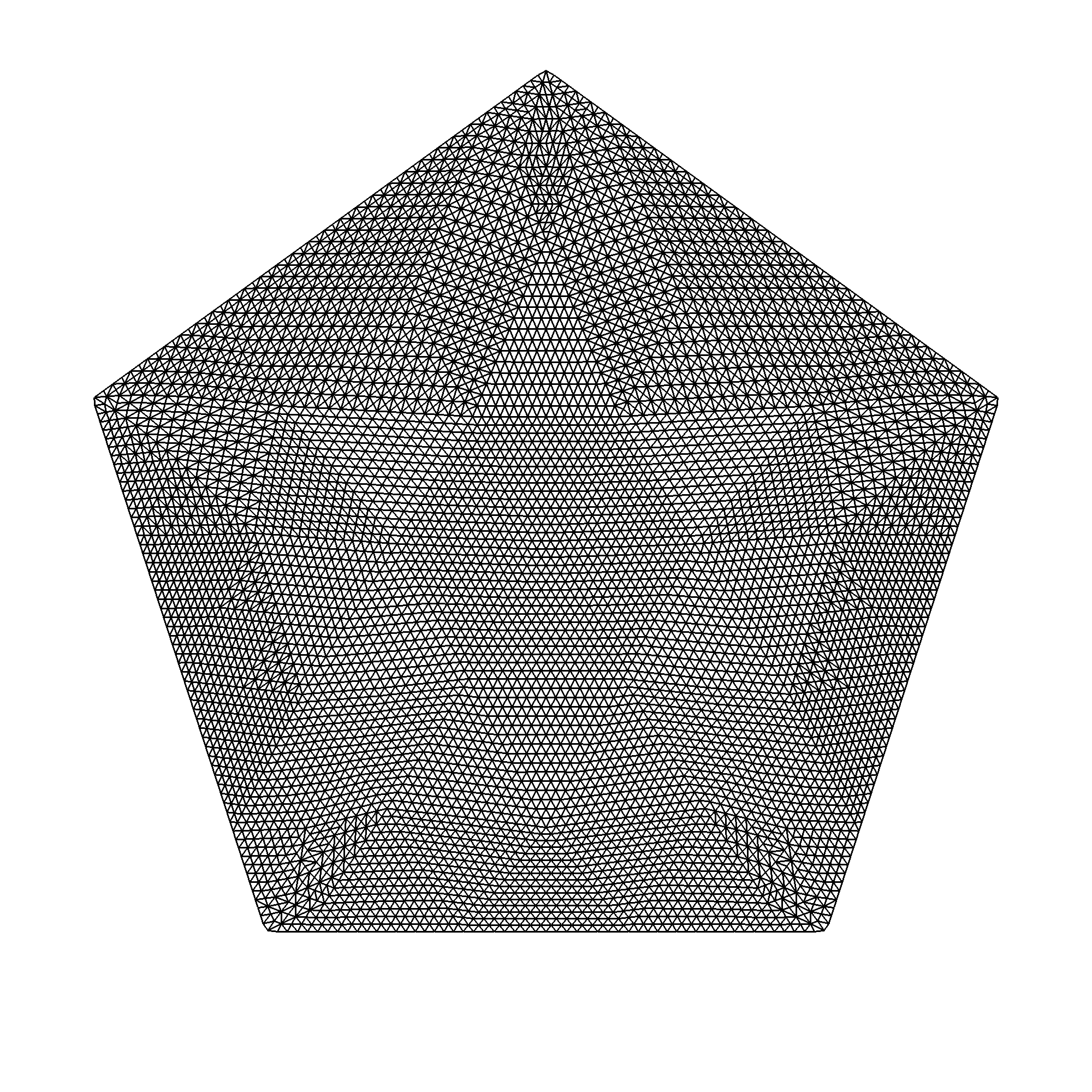}%
	\caption{Initial domain and coarse triangulation (upper left plot) and plots of optimal shapes 
        for refined triangulations for the example discussed in \cref{subsec:ex2}. Clearly defined kinks 
        occur on the boundaries.}
	\label{fig:nonsmooth}
\end{figure}
The
overall runtime was 4 minutes
and
the solution times of the QP \eqref{eq:discrete_riesz_representation_4}
are similar to previous example.
The number of iterations per refinement level
are
$19$, $20$, $18$, $15$ and $14$, respectively.
% Note that these numbers are significantly slower

\subsection{Non-convergence in 3 dimensions}
\label{subsec:ex3}
As addressed briefly at the end of \cref{subsec:feasible_def_fields},
it is possible to apply \cref{alg:alg} also to three-dimensional shape
optimization problems with convexity constraint. Our third example shows
however that a direct discretization of the constraint can lead to 
non-convergence. Instead of Dirichlet boundary conditions we impose 
Neumann boundary conditions and include a lower order term in the state
equation, i.e., the modified shape optimization problem reads as follows:
\begin{equation*}
	\begin{aligned}
		\text{Minimize}\quad & \int_\Omega u(x) \, \dx \\
		\text{w.r.t.}\quad & \Omega \subset \R^3, u \in H^1(\Omega) \\
		\text{s.t.}\quad&
		-\Delta u + u = f\text{ in }\Omega, \quad \frac{\partial u}{\partial n} = 0 \text{ on }\partial\Omega, \\
		& \Omega \text{ convex and open}
	\end{aligned}
\end{equation*}
We chose $f(x) = x_1^2 + x_2^2 + x_3^2 - 1$.
Since the problem is rotationally symmetric, we expect that the minimizer is a ball,
which is obviously convex. We start with a tetrahedral grid on the cube
$[-1/2, 1/2]^3$, see the top left plot of \cref{fig:convex_sphere}
and solved the discretized problem on three different mesh levels.
The computational times was around 2 hours
and the solution of the QP \eqref{eq:discrete_riesz_representation_4}
took several minutes on the last mesh.
The numbers of iterations were
16, 14, 23, respectively.
The solutions are presented in \cref{fig:convex_sphere}.
\begin{figure}[p]
	\centering
	\includegraphics[width=.35\textwidth]{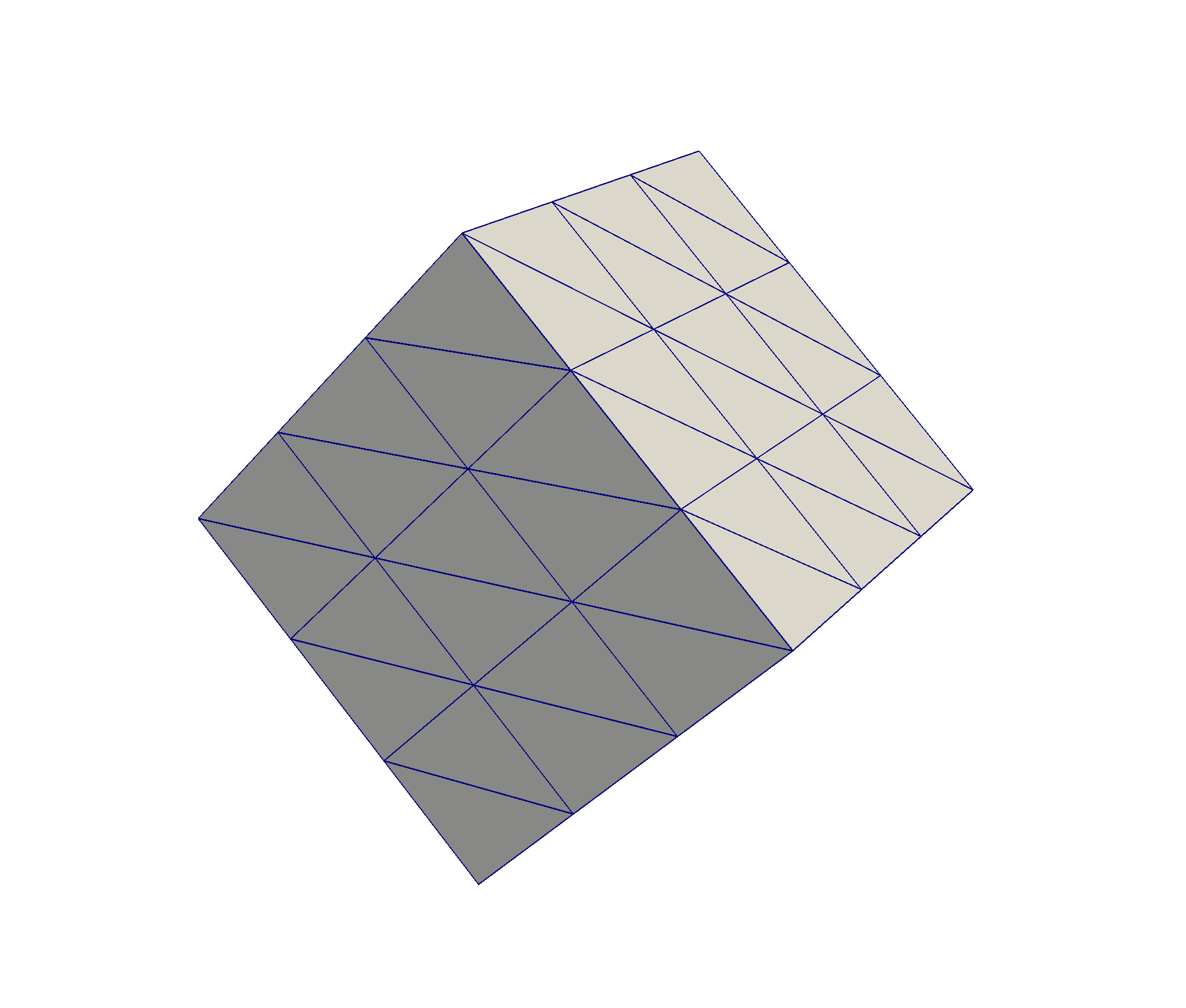}\hspace*{.1\linewidth}%
	\includegraphics[width=.35\textwidth]{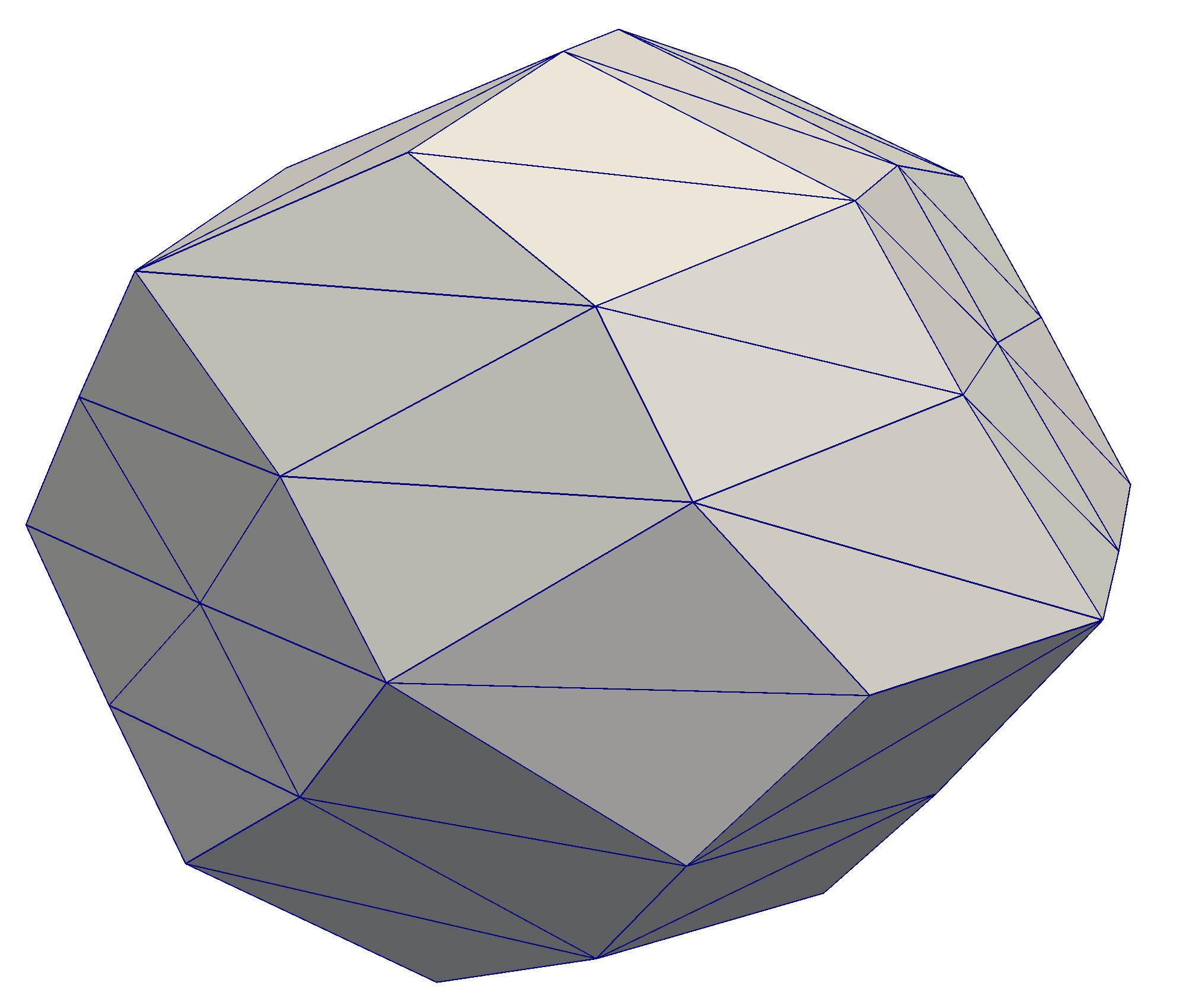}\\
	\includegraphics[width=.35\textwidth]{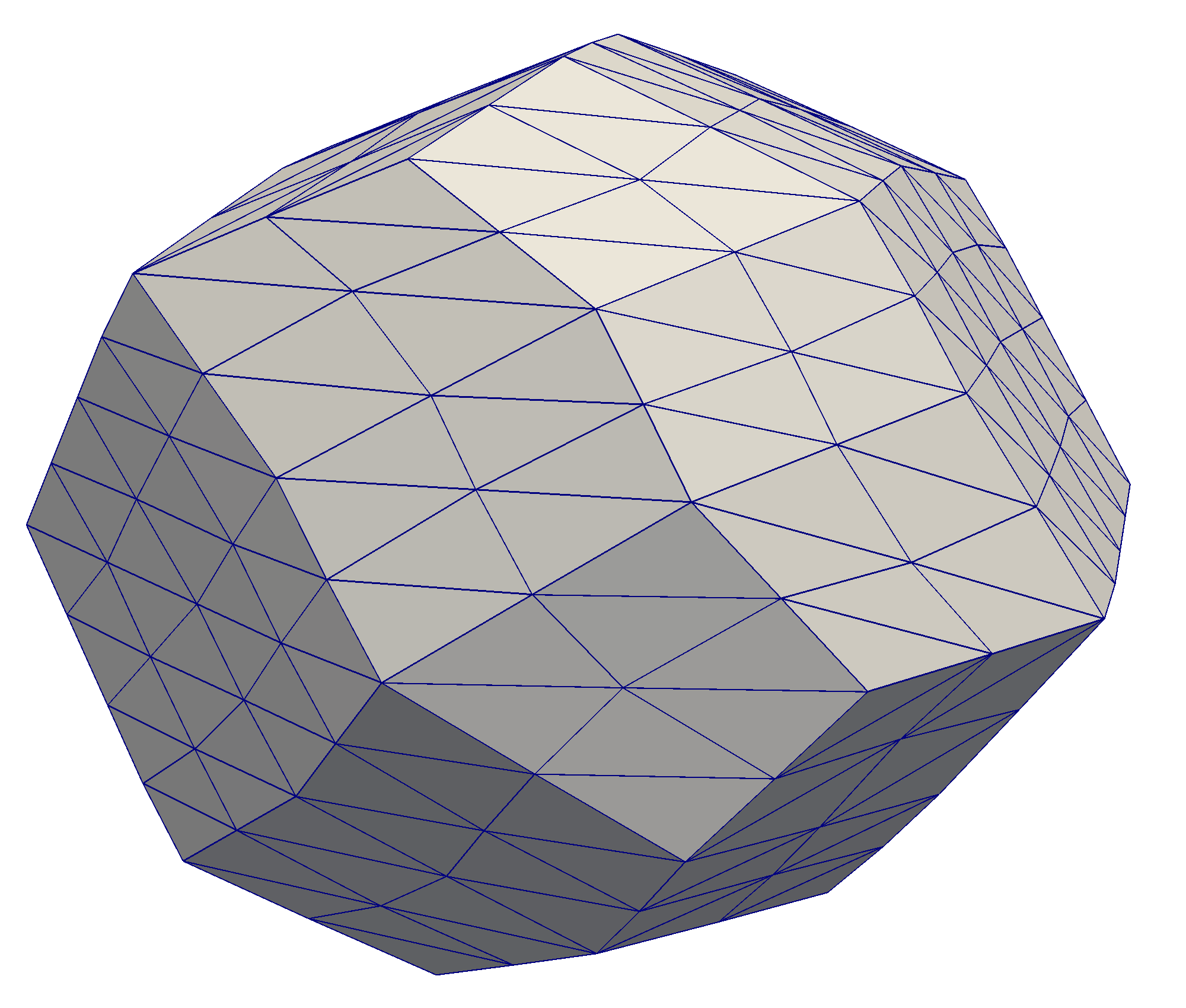}\hspace*{.1\linewidth}%
	\includegraphics[width=.35\textwidth]{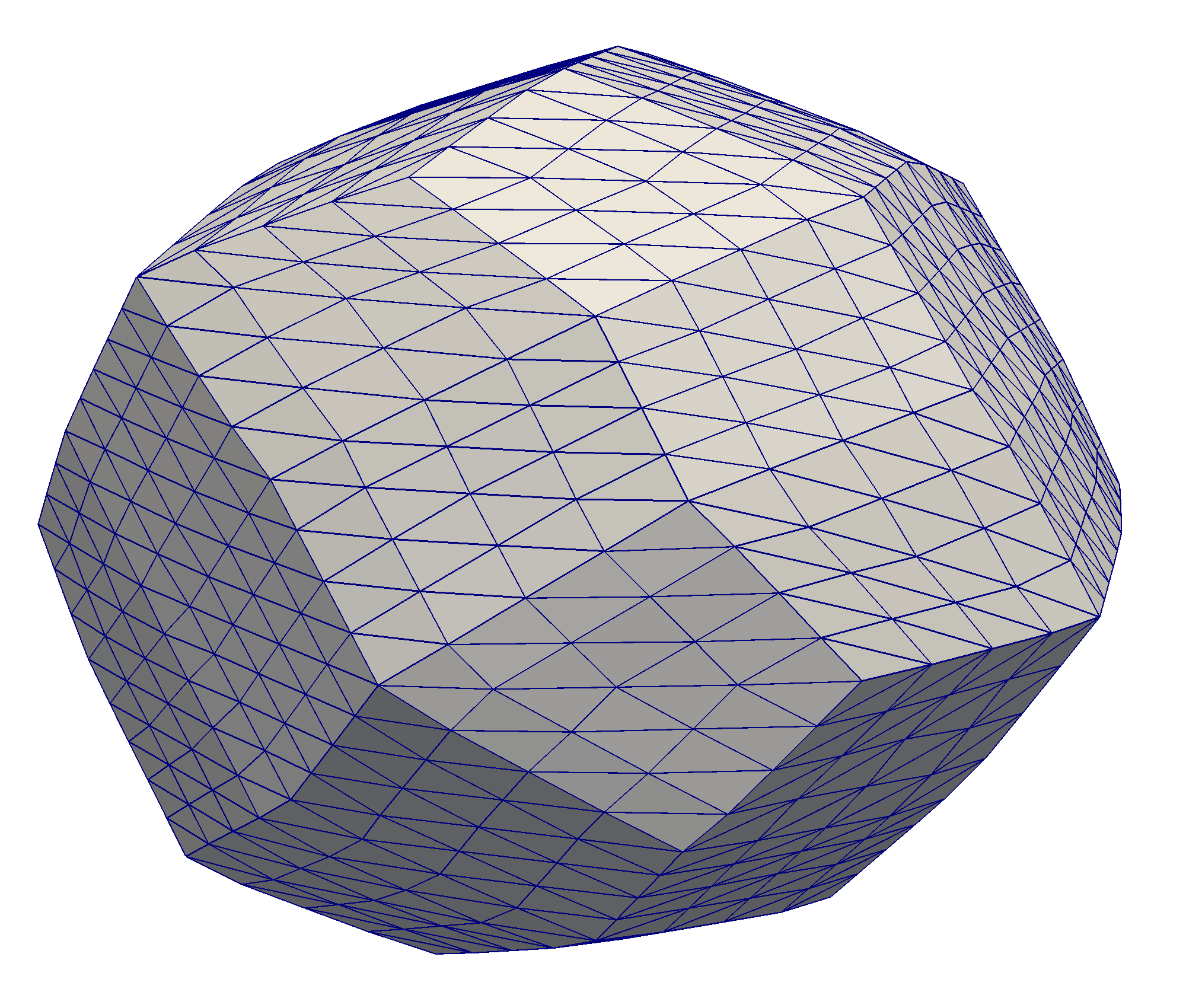}%
	\caption{Numerical solutions of the problem of \cref{subsec:ex3} with convexity constraints: 
        the rotational symmetry of the problem is strongly violated.}
	\label{fig:convex_sphere}
\end{figure}
As it can be seen in this figure,
the solutions are not rotationally symmetric and contain
many flat parts on their boundary.

If we solve the same problem, but without the convexity constraint,
we arrive at the bodies presented in \cref{fig:nonconvex_sphere}.
\begin{figure}[p]
	\centering
	\includegraphics[width=.35\textwidth]{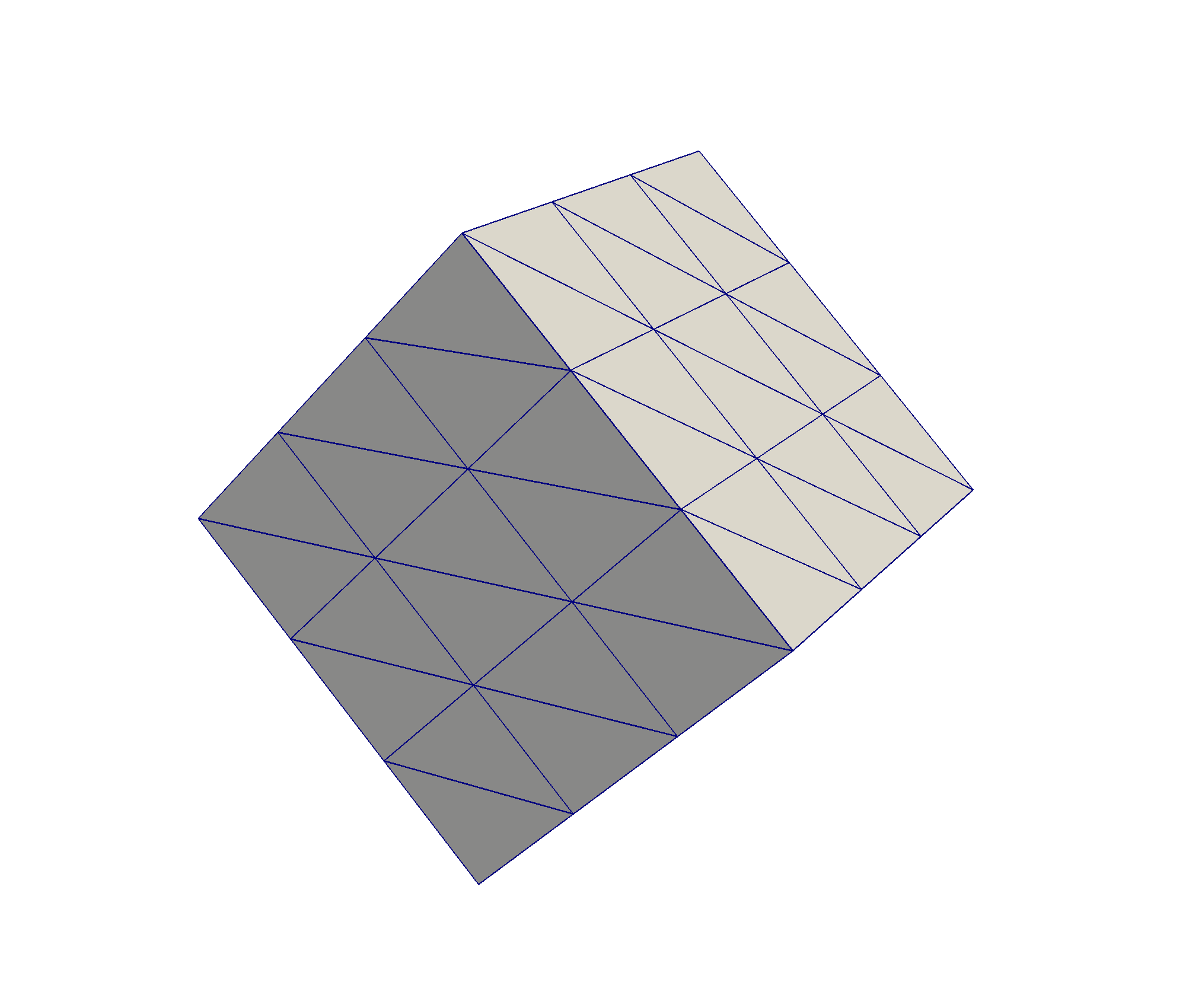}\hspace*{.1\linewidth}%
	\includegraphics[width=.35\textwidth]{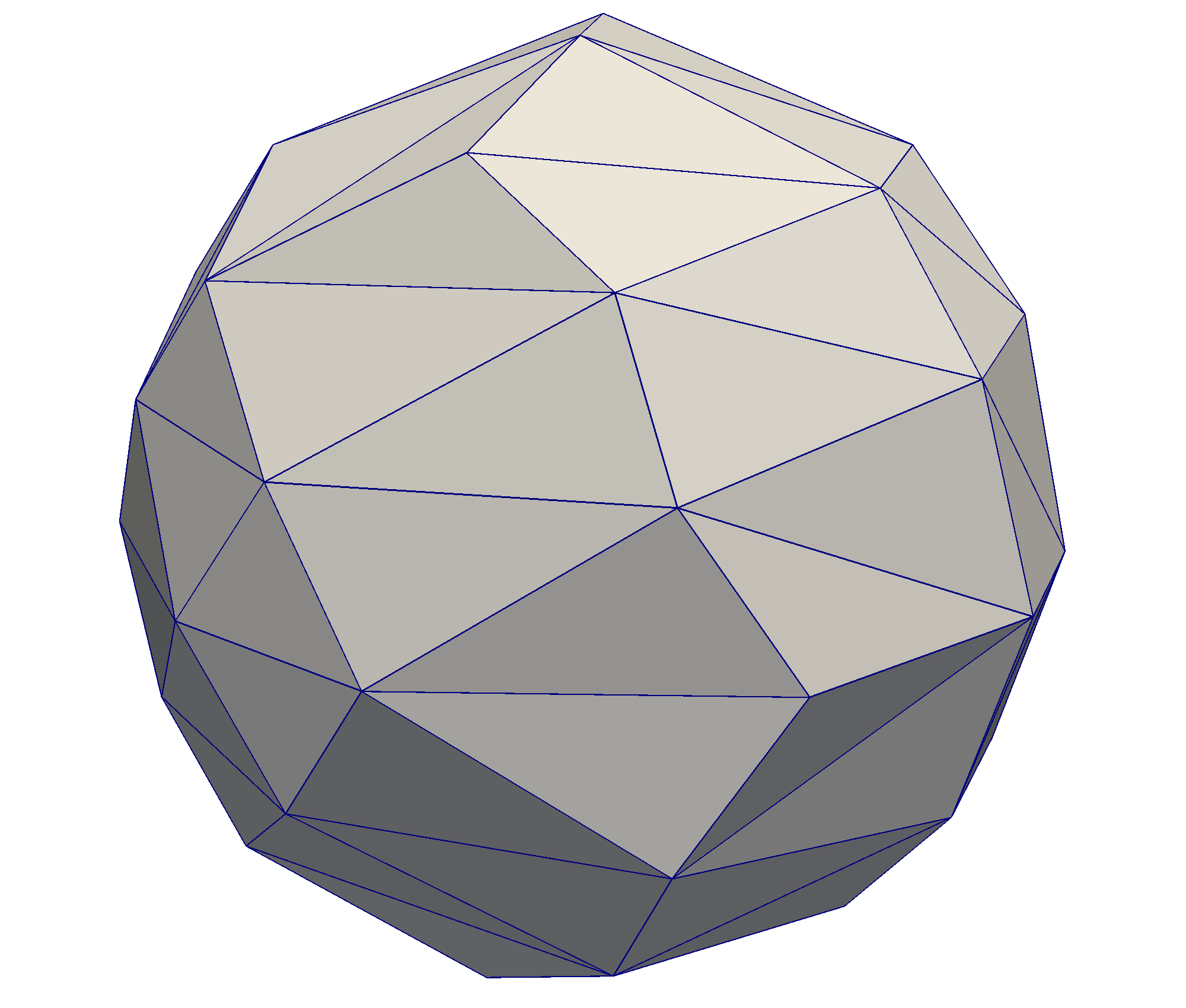}\\
	\includegraphics[width=.35\textwidth]{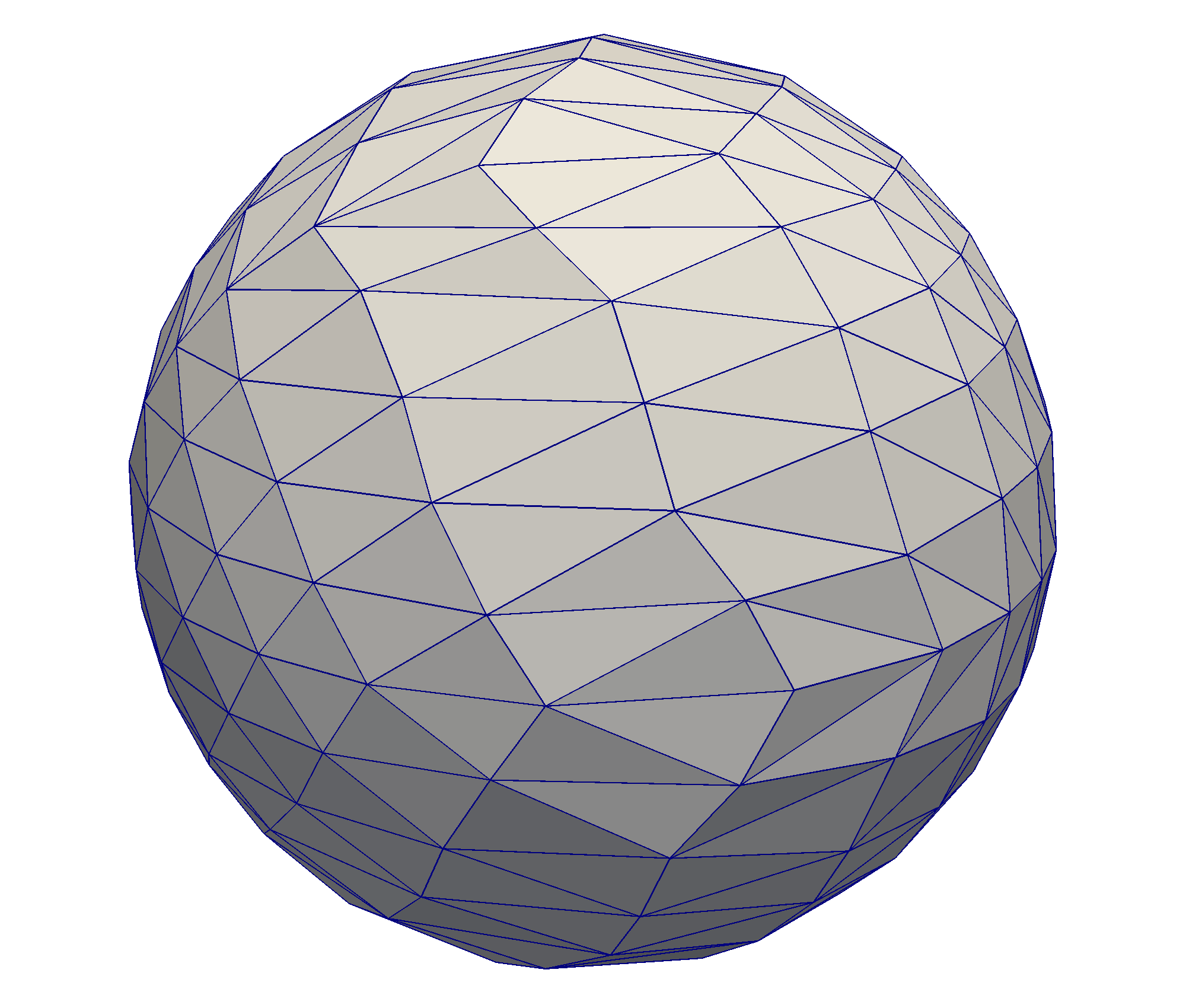}\hspace*{.1\linewidth}%
	\includegraphics[width=.35\textwidth]{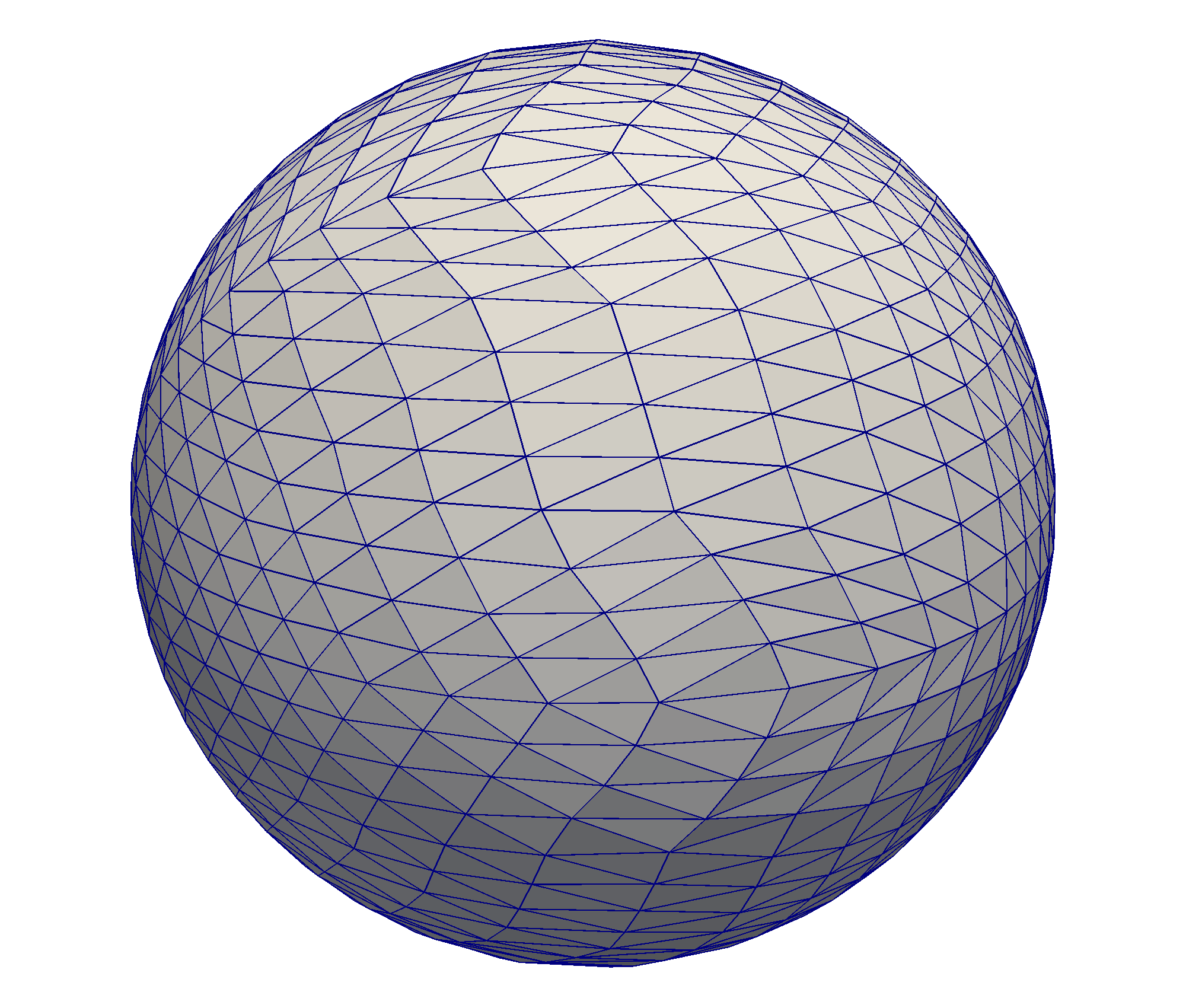}%
	\caption{Numerical solutions of the problem of \cref{subsec:ex3} without convexity constraints:
	discrete optimal shapes are nearly rotationally symmetric but not convex.}
	\label{fig:nonconvex_sphere}
\end{figure}
In these plots,
one can see that the optimal solution of the above problem
might be indeed a ball.
However, a closer inspection of \cref{fig:nonconvex_sphere}
reveals that the approximations are not convex,
there are some ``non-convex parts'' on the boundary.

Thus, the bodies in \cref{fig:convex_sphere}
are not approximating the optimal shape
and this shows that the convergence results of~\cref{sec:convergence}
do not carry over to the three-dimensional situation.

This non-convergence phenomenon is well known for the approximation of convex functions.
Indeed, in two dimensions it is not possible to approximate arbitrary convex functions
by convex finite element functions of order one on, e.g., uniform refinements of a fixed grid,
see \cite{ChoneLeMeur2001}.
However, for functions of order at least two,
there exist approximation results, see \cite{AguileraMorin2009,Wachsmuth2016:1}.

%%fakesection: Acknowledgement
\subsection*{Acknowledgments}
This work is supported by DFG grants within the
\href{https://spp1962.wias-berlin.de}{Priority Program SPP 1962}
(Non-smooth and Complementarity-based Distributed Parameter Systems: Simulation and Hierarchical Optimization).

%%fakesection: Bibliography
\ifbiber
	\printbibliography
\else
	\bibliographystyle{plainnat}
	\bibliography{references}
\fi
\end{document}

\section{Open problems}
Let $\Omega$ be an open, convex set.
Is there a nice characterization
of the set of admissible perturbations
\begin{equation*}
	\VV_{\textup{ad}}
	=
	\{ V \mid \text{the sets } (I + t\,V)(\Omega) \text{ are convex for small } t \}
\end{equation*}
?
With this set, one could formulate the stationarity conditions
\begin{equation*}
	J'(\Omega; V) \ge 0
	\qquad\forall V \in \VV_{\textup{ad}}
\end{equation*}

Let $\Omega_h$ be an ``almost'' convex polygon.
Is it possible to ``project'' this polygon into the set of convex polygons?
What is a possible metric for defining this projection?

Is it possible to use a second-order algorithm
(maybe similar to an SQP approach)?

Characterize convexity via mean curvature,
characterize feasible deformations $V$ via derivative of mean curvature

How to solve the QP \eqref{eq:discrete_riesz_representation_4}
efficiently?

Isoparametric P2? \cite{Wachsmuth2016:1}.